\documentclass[a4paper,10pt,leqno]{amsart}
\title[Recipes to compute the algebraic $K$-theory of Hecke algebras]
  {Recipes to compute the algebraic $K$-theory of Hecke algebras  of reductive $p$-adic groups}

\author{Bartels, A.}
\address{WWU M\"unster\\
  Mathematisches Institut\\
  Einsteinstr.~62, D-48149 M\"unster, Germany} \email{a.bartels@uni-muenster.de}
\urladdr{http://www.math.uni-muenster.de/u/bartelsa} \author{L\"uck, W.}
\address{Mathematisches Institut der Universit\"at Bonn\\
  Endenicher Allee 60\\
  53115 Bonn, Germany} \email{wolfgang.lueck@him.uni-bonn.de}
\urladdr{http://www.him.uni-bonn.de/lueck} \date{March 2024}
          \makeatletter
     \@namedef{subjclassname@2020}{%
  \textup{2020} Mathematics Subject Classification}
\makeatother

  \usepackage[utf8]{inputenc}
  \usepackage{hyperref}
  \usepackage{comment}
  \usepackage{calc}
  \usepackage{enumerate,amssymb,bm}
  \usepackage[arrow,curve,matrix,tips,2cell]{xy}
    \SelectTips{eu}{10} \UseTips
    \UseAllTwocells
  \usepackage{tikz}
  \usetikzlibrary{decorations.pathreplacing}
  \usepackage{color}
  \usepackage{scalerel}
  \usepackage{braket}
  \usepackage{mathtools}
  \usepackage{bbm}
  \usepackage{enumitem}
  \usepackage{float}
  \DeclareMathAlphabet{\matheurm}{U}{eur}{m}{n}

\DeclareMathAlphabet{\matheurm}{U}{eur}{m}{n}

\newcommand{\Chcata}[1]{#1\text{-}\matheurm{Ch}}

\newcommand{\MODcat}[1]{#1\text{-}\matheurm{Mod}}

\newcommand{\Or}{\matheurm{Or}}

\newcommand{\OrGF}[2]{\matheurm{Or}_{#2}(#1)}

\newcommand{\Spaces}{\matheurm{Spaces}}

\newcommand{\Spectra}{\matheurm{Spectra}}

\newcommand{\Sub}{\matheurm{Sub}}
\newcommand{\SubGF}[2]{\matheurm{Sub}_{#2}(#1)}

\DeclareMathOperator{\cent}{cent}

\DeclareMathOperator{\cok}{cok}
\DeclareMathOperator{\colim}{colim}
\DeclareMathOperator*{\colimunder}{colim}

\DeclareMathOperator{\conhom}{conhom}

\DeclareMathOperator{\GL}{GL}

\DeclareMathOperator*{\hocolimunder}{hocolim}

\DeclareMathOperator{\id}{id}

\DeclareMathOperator{\map}{map}
\DeclareMathOperator{\mor}{mor}

\DeclareMathOperator{\PGL}{PGL}

\DeclareMathOperator{\pr}{pr}

\DeclareMathOperator{\sh}{sh}

\DeclareMathOperator{\sing}{sing}
\DeclareMathOperator{\SL}{SL}

\newcommand{\COM}{{\calc\hspace{-1pt}\mathrm{om}}}

\newcommand{\COP}{{\calc\hspace{-1pt}\mathrm{op}}}

\newcommand{\OPEN}{{\calo\mathrm{p}}}

  \newcommand{\IC}{\mathbb{C}}

  \newcommand{\IQ}{\mathbb{Q}}
  \newcommand{\IR}{\mathbb{R}}

  \newcommand{\IZ}{\mathbb{Z}}

  \newcommand{\calc}{\mathcal{C}}

  \newcommand{\calf}{\mathcal{F}}

  \newcommand{\calh}{\mathcal{H}}
  \newcommand{\cali}{\mathcal{I}}

  \newcommand{\calo}{\mathcal{O}}
  
  \newcommand{\calp}{\mathcal{P}}

  \newcommand{\bfK}{\mathbf{K}}

\newcommand{\op}{{\mathrm{op}}}

\newcommand{\EGF}[2]{E_{#2}(#1)}

\newcounter{commentcounter}

\theoremstyle{plain}
\newtheorem{theorem}{Theorem}[section]

\newtheorem{lemma}[theorem]{Lemma}

\newtheorem{condition}[theorem]{Condition}
\newtheorem*{theorem*}{Theorem}
\newtheorem*{theoremA*}{Theorem A}
\newtheorem*{theoremB*}{Theorem B}

\theoremstyle{definition}
\newtheorem{definition}[theorem]{Definition}

\newtheorem{example}[theorem]{Example}

\newtheorem{remark}[theorem]{Remark}

\newtheorem*{definition*}{Definition}

\theoremstyle{remark}

\makeatletter\let\c@equation=\c@theorem\makeatother

\theoremstyle{definition}

\newcounter{othercommentcounter}

\newcommand{\version}[1]              
{\begin{center} last edited on #1\\
last compiled on \today\\
name of tex-file: \jobname
\end{center}
}

\begin{document}

\maketitle

\begin{abstract}
  We compute the algebraic $K$-theory of the Hecke algebra of a reductive $p$-adic group
  $G$ using the fact that the Farrell-Jones Conjecture is known in this context. The main
  tool will be the properties of the associated Bruhat-Tits building and an equivariant
  Atiyah-Hirzebruch spectral sequence.  In particular the projective class group can be
  written as the colimit of the projective class groups of the compact open subgroups of
  $G$.
\end{abstract}


\section{Introduction}\label{sec:Introduction}
\typeout{---------------------------------  Section 1:  Introduction ------------------------}


We begin with stating the main theorem of this paper, explanation will follow:

\begin{theorem}[Main Theorem]\label{the:Main_Theorem}\
  Let $G$ be a td-group which is modulo a normal compact subgroup a subgroup of a
  reductive $p$-adic group.  Let $R$ be a uniformly regular ring with $\IQ \subseteq R$.
  Choose a model $\EGF{G}{\COP}$ for the
  classifying space for  proper smooth $G$-actions. Let $\cali \subseteq \COP$
  be the set of isotropy groups of points in $\EGF{G}{\COP}$.

  Then
  \begin{enumerate}

  \item\label{the:Main_Theorem:assembly}
    The  map induced by the projection $\EGF{G}{\COP} \to G/G$ induces
  for every $n \in \IZ$ an isomorphism
  \[
H_n^G(\EGF{G}{\COP};\bfK_R) \to H_n^G(G/G;\bfK_R) = K_n(\calh(G;R));
\]

\item\label{the:Main_Theorem:spectral_sequence}
There is a (strongly convergent) spectral sequence
\[
  E_{p,q}^2 = S\!H^{G,\cali}_{p}\bigl(\EGF{G}{\COP};\overline{K_q(\calh(?;R))}\bigr)
  \implies K_{p+q}(\calh(G;R)),
\]
whose $E^2$-term is concentrated in the first quadrant;

  \item\label{the:Main_Theorem:K_0}

  The canonical map induced by the various inclusions $K \subseteq G$
  \[
   \colimunder_{K \in \SubGF{G}{\cali}} K_0 (\calh(K;R)) \to K_0 (\calh(G;R))
\]
can be identified with the isomorphism appearing in
assertion~\ref{the:Main_Theorem:assembly} in degree $n = 0$ and hence is bijective;
  \item\label{the:Main_Theorem:negative}
    We have $K_{n}(\calh(G;R)) = 0$  for $n \le -1$.
  \end{enumerate}
\end{theorem}

Note that assertion~\ref{the:Main_Theorem:assembly} of
Theorem~\ref{the:Main_Theorem} is proved
in~\cite[Corollary~1.8]{Bartels-Lueck(2023K-theory_red_p-adic_groups)}.
So this papers deals with implications of it concerning
computations of the algebraic $K$-groups $K_n(\calh(G))$ of the Hecke
algebra of $G$.

A \emph{td-group $G$} is a  locally compact second countable totally disconnected
topological Hausdorff group.  It is \emph{modulo a normal compact subgroup a subgroup of a
  reductive $p$-adic group} if it contains a (not necessarily open) normal compact
subgroup $K$ such that $G/K$ is isomorphic to a subgroup of some reductive $p$-adic group.

A ring is called \emph{uniformly regular}, if it is Noetherian and there exists a natural
number $l$ such that any finitely generated  $R$-module admits a resolution by
projective $R$-modules of length at most $l$.  We write $\IQ \subseteq R$, if for any
integer $n$ the element $n \cdot 1_R$ is a unit in $R$.  Examples for uniformly regular
rings $R$ with $\IQ \subseteq R$ are fields of characteristic zero.

We denote by $\calh(G;R)$ the \emph{Hecke algebra} consisting of locally constant functions
$s \colon G \to R$ with compact support, where the additive structure comes from the
additive structure of $R$ and the multiplicative structure from the convolution
product. Note  that $\calh(G;R)$ is a ring without unit.

We denote by $\EGF{G}{\COP}$ a model for the \emph{classifying space for proper smooth
  $G$-actions}, i.e., a $G$-$CW$-complex, whose isotropy groups are all compact open
subgroups of $G$ and for which $\EGF{G}{\COP}^H$ is weakly contractible for any compact
open subgroup $H \subseteq G$. Two such models are $G$-homotopy equivalent. Hence
$H_n^G(\EGF{G}{\COP};\bfK_R)$ is independent of the choice of a model. If $G$ is a
reductive $p$-adic group with compact center, then its Bruhat-Tits building is a model for
$\EGF{G}{\COP}$. If the center is not compact, one has to pass to the extended
Bruhat-Tits building.

We will construct a \emph{smooth $G$-homology theory} $H_*^G(-;\bfK_R)$ in
Section~\ref{sec:A_brief_review_of_the_farrell_Jones_Conjecture_for_the_algebraic_K-theory_of_Hecke_categories}.
It assigns to a smooth $G$-$CW$-pair $(X,A)$ a collection of abelian groups
$\calh^G_n(X,A;\bfK_R)$ for $n \in \IZ$ that satisfies the expected axioms, i.e., long
exact sequence of a pair, $G$-homotopy invariance, excision, and the disjoint union axiom.
Moreover, for every open subgroup
$U \subseteq G$ and $n \in \IZ$ we have
\begin{equation}
H_n^G(G/U;\bfK_R) \cong K_n(\calh(U;R)).
\label{computation_of_H_n_upper_G(G/U;bfK_R)}  
\end{equation}

Let $\calf$ be a collection of open subgroups of $G$ which is closed under conjugation.
Examples are the set $\COP$ of compact open subgroups of $G$  and the set $\cali$ of isotropy groups of points
 of some model for $\EGF{G}{\COP}$. The subgroup category $\SubGF{G}{\calf}$ appearing in
Theorem~\ref{the:Main_Theorem}~\ref{the:Main_Theorem:K_0} has $\calf$ as set of objects
and will be described in detail in
Subsection~\ref{subsec:The_smooth_orbit_category_and_the_smooth_subgroup_category}.

The abelian groups
$S\!H^{G,\calf}_{p}\bigl(\EGF{G}{\calf};\overline{K_q(\calh(?;R))}\bigr)$ appearing in
Theorem~\ref{the:Main_Theorem}~\ref{the:Main_Theorem:spectral_sequence} will be defined for the
covariant functor $\overline{K_q(\calh(?;R))} \colon \SubGF{G}{\calf} \to \MODcat{\IZ}$,
whose value at $U \in \calf$ is $K_n(\calh(U;R))$,  in
Subsection~\ref{subsec:Cellular_chain_complexes_and_Bredon_homology_smooth}. They are
closely related to the \emph{Bredon homology groups}
$B\!H^{G,\calf}_{p}\bigl(\EGF{G}{\calf};K_q(\calh(?;R))\bigr)$.

The proof of the Main Theorem~\ref{the:Main_Theorem} will be given in
Section~\ref{sec:Proof_of_the_Main_Theorem_ref(the:Main_Theorem)}.

The relevance of the Hecke algebra $\calh(G;R)$ is that the category of non-degenerate
modules over it is isomorphic to the category of smooth $G$-representations with
coefficients in $R$, see for instance~\cite{Bernstein(1992),Garrett(2012)}. Hence in
particular its projective class group $K_0(\calh(G;R))$ is important. The various
inclusions $K \to G$ for $K \in \COP$ induce a map
\begin{equation}
  \bigoplus_{K \in \COP} K_0(\calh(K;R))  \to K_0(\calh(G;R)),
  \label{induction_direct_sums}
\end{equation}
which factorizes over the canonical epimorphism
from $\bigoplus_{K \in \COP} K_0(\calh(K;R))$ to \linebreak $\colimunder_{K \in \SubGF{G}{\cali}} K_0 (\calh(K;R))$
to the isomorphism appearing in
Theorem~\ref{the:Main_Theorem}~\ref{the:Main_Theorem:K_0} and is hence surjective.
Dat~\cite{Dat(2007)} has shown that the map~\eqref{induction_direct_sums} is rationally
surjective for $G$ a reductive $p$-adic group and $R=\IC$.  In particular, the cokernel of
it is a torsion group.  Dat~\cite[Conj.~1.11]{Dat(2003)} conjectured that this cokernel is
$\widetilde w_G$-torsion.  Here $\widetilde w_G$ is a certain multiple of the order of the
Weyl group of $G$.  Dat~\cite[Prop.~1.13]{Dat(2003)}  proved this conjecture for
$G = \GL_n(F)$ for a $p$-adic field $F$ of characteristic zero and asked about the integral version, see the
comment following~\cite[Prop.~1.10]{Dat(2003)}, which is now proven by
Theorem~\ref{the:Main_Theorem}~\ref{the:Main_Theorem:K_0}.

The computations simplify considerably in the case of a reductive $p$-adic group thanks
to the associated (extended) Bruhat-Tits building, see
Sections~\ref{sec:The_main_recipe_for_the_computation_of_the_projective_class_group}
and~\ref{sec:homotopy-colimits}. As an illustration we analyze the projective class groups
of the Hecke algebras of $\SL_n(F)$, $\PGL_n(F)$ and $\GL_n(F)$ in
Section~\ref{sec:K_0-SL-Gl-PGL}.

One of our main tools will be the \emph{smooth equivariant Atiyah-Hirzebruch spectra
  sequence}, which we will establish and examine in
Section~\ref{sec:The_smooth_equivariant_Atiyah-Hirzebruch_spectral_sequence}.


\subsection*{Acknowledgments}

We thank Eugen Hellmann and Linus Kramer for many helpful comments and discussions.

We thank the referee for the careful and helpful report.

The paper is funded by the ERC Advanced Grant \linebreak ``KL2MG-interactions'' (no.
662400) of the second author granted by the European Research Council, by the Deutsche
Forschungsgemeinschaft (DFG, German Research Foundation) under Germany's Excellence
Strategy \--- GZ 2047/1, Projekt-ID 390685813, Hausdorff Center for Mathematics at Bonn,
and by the Deutsche Forschungsgemeinschaft (DFG, German Research Foundation) under
Germany's Excellence Strategy EXC 2044 \--- 390685587, Mathematics M\"unster: Dynamics
\--- Geometry \--- Structure.

The paper is organized as follows:

\tableofcontents


\typeout{---------- Section 2: The smooth equivariant Hirzebruch spectral sequence ------------------}

\section{The smooth equivariant Hirzebruch spectral sequence}%
\label{sec:The_smooth_equivariant_Atiyah-Hirzebruch_spectral_sequence}

Throughout this section we fix a set $\calf$ of open subgroups of $G$ which is closed
under conjugation. Our main examples for $\calf$ are the family $\OPEN$ of all open
subgroups and the family $\COP$ of all compact open subgroups.
A \emph{$\calf$-$G$-$CW$-complex} $X$ is a $G$-$CW$-complex $X$
such that for every $x \in X$ its isotropy group $G_x$ belongs to $\calf$.
A smooth $G$-$CW$-complex is the same as a $\OPEN$-$CW$-$CW$-complex
and a proper smooth $G$-$CW$-complex is the same as a $\COP$-$CW$-complex.
Let $\calh_*^G$ be a smooth $G$-homology theory.

The main result of this section is

\begin{theorem}\label{the:equivariant_Atyiah-Hirzebruch_spectral_sequence_smooth}
  Consider a  pair $(X,A)$ of $\calf$-$G$-$CW$-complexes and a smooth $G$-homology theory $\calh^G_*$.
  Then there is an equivariant Atyiah-Hirzebruch spectral sequence
  converging to $\calh_{p+q}^G(X,A)$, whose $E^2$-term is given by
    \[
      E^2_{p,q} = B\!H^{G,\calf}_p(X,A;\calh^G_q)
    \]
    for the Bredon homology $B\!H^{G,\calf}_p(X,A;\calh^G_q)$ of $(X,A)$ with coefficients in the
    covariant $\IZ\OrGF{G}{\calf}$-module $\calh^G_q$ that sends $G/H$ to $\calh^G_q(G/H)$.
  \end{theorem}

  The remainder of this section is devoted to the definition of the Bredon homology, the
  construction of the equivariant Atiyah-Hirzebruch spectral sequence, and some general
  calculations concerning the $E^2$-term. Convergence means that there is an ascending
  filtration $F_{l,m-l}\calh_{m}^G(X,A)$ for $l = 0,1,2, \ldots$ of $\calh_{m}^G(X,A)$
  such that
  $F_{p,q}\calh_{p+q}^G(X,A)/F_{p-1,q+1}\calh_{p+q}^G(X,A) \cong E^{\infty}_{p,q}$ holds
  for $E^{\infty}_{p,q} = \colim_{r \to \infty} E^r_{p,q}$.

  
\subsection{The smooth orbit category and the smooth subgroup category}%
\label{subsec:The_smooth_orbit_category_and_the_smooth_subgroup_category}

The \emph{$\calf$-orbit category} $\OrGF{G}{\calf}$ has as objects homogeneous $G$-spaces $G/H$
with $H \in \calf$.
Morphisms from $G/H$ to $G/K$ are  $G$-maps $G/H \to G/K$.
We will put no topology on $\OrGF{G}{\calf}$. For
any $G$-map $f \colon G/H \to G/K$ of smooth homogeneous spaces,
there is an element $g \in G$ such that
$gHg^{-1} \subseteq K$ holds and $f$ is the $G$-map $R_{g^{-1}} \colon G/H \to G/K$ sending
$g'H$ to $g'g^{-1}K$.  Given two elements $g_0,g_1 \in G$ such that
$g_iHg_i^{-1} \subseteq K$ holds for $i = 0,1$, we have
$R_{g_0^{-1}} = R_{g_1^{-1}} \Longleftrightarrow g_1g_0^{-1} \in K$.  We get a
bijection
\begin{equation}
  K\backslash \{g \in G \mid gHg^{-1} \subseteq K\}
  \xrightarrow{\cong} \map_G(G/H,G/K), \quad g \mapsto R_{g^{-1}}.
  \label{identification_of_map_G(G/H,G/K)}
\end{equation}

The \emph{$\calf$-subgroup category} $\SubGF{G}{\calf}$ has $\calf$ as the set of
objects. For $H,K \in \calf$ denote by $\conhom_G(H,K)$ the set of group
homomorphisms $f\colon  H \to K$, for which there exists an element $g \in G$ with
$gHg^{-1} \subset K$ such that $f$ is given by conjugation with $g$, i.e.,
$f = c(g): H \to K, \hspace{3mm} h \mapsto ghg^{-1}$.  Note that $c(g) = c(g')$ holds
for two elements $g,g' \in G$ with $gHg^{-1} \subset K$ and $g'Hg'^{-1} \subset K$, if
and only if $g^{-1}g'$ lies in the centralizer
$C_GH = \{g \in G \mid gh=hg \mbox{ for all } h \in H\}$ of $H$ in $G$. The group of inner
automorphisms $\operatorname{Inn}(K)$ of $K$ acts on $\conhom_G(H,K)$ from the left by
composition. Define the set of morphisms
\[
\mor_{\Sub_\COP(G)}(H,K) := \operatorname{Inn}(K) \backslash \conhom_G(H,K).
\]
There is an obvious bijection
\begin{multline}
  K\backslash \{g \in G \mid gHg^{-1} \subseteq K\}/C_GH \xrightarrow{\cong}
  \operatorname{Inn}(K) \backslash \conhom_G(H,K), \\
  \quad KgC_GH \mapsto [c(g)],
  \label{identification_of_K_backslash_(g_in_G_mid_gHg_upper_(-1)_subseteq_K)/C_GH}
\end{multline}
where $[c(g)] \in \operatorname{Inn}(K) \backslash \conhom_G(H,K)$
  is the class represented by the element
  $c(g) \colon H \to K, \; h \mapsto ghg^{-1}$ in $\conhom_G(H,K)$ and
  $K$ acts from the left and $C_GH$ from the right on $\{g \in G \mid gHg^{-1} \subseteq K\}$
  by the multiplication in $G$.

  Let
  \begin{equation}
    P\colon \OrGF{G}{\calf} \to \SubGF{G}{\calf}
    \label{projection_Or(G)_to_Sub(G)}
  \end{equation}
  be the canonical projection which sends an object
  $G/H$ to $H$ and is given on morphisms by the obvious projection under the
  identifications~\eqref{identification_of_map_G(G/H,G/K)}
  and~\eqref{identification_of_K_backslash_(g_in_G_mid_gHg_upper_(-1)_subseteq_K)/C_GH}.


\subsection{Cellular chain complexes and Bredon homology}%
\label{subsec:Cellular_chain_complexes_and_Bredon_homology_smooth}

Given an $\calf$-$G$-$CW$-complex $X$, we obtain a contravariant $\OrGF{G}{\calf}$-space
$O_X \colon \OrGF{G}{\calf} \to \Spaces$ by sending $G/H$ to $\map_G(G/H,X) = X^H$. We get a
contravariant $\SubGF{G}{\calf}$-space $S_X \colon \SubGF{G}{\calf} \to \Spaces$ by
sending $H$ to $C_GH\backslash \map_G(G/H,X) = C_GH\backslash X^H$.  A morphism $H \to K$
given by an element $g \in G$ satisfying $gHg^{-1} \subseteq K$ is sent to the map
$C_GK\backslash X^K \to C_GH \backslash X^H$ induced by the map
$X^K \to X^H, \; x \mapsto g^{-1}x$.

Given a pair $(Y,A)$ with a filtration
$A = Y_{-1} \subseteq Y_0 \subseteq Y_1 \subseteq Y_2\subseteq \cdots \subseteq Y$ with
$Y = \colim_{n \to \infty} Y_n$, we associate to it a $\IZ$-chain complex $C^c_*(Y,A)$,
whose $n$-th chain module is the singular homology $H^{\sing}_n(Y_n,Y_{n-1})$ of the
pair $(Y_n,Y_{n-1})$ (with coefficients in $\IZ$) and whose $n$th differential is given by
the composite
\[H_n^{\sing} (Y_n,Y_{n-1}) \xrightarrow{\partial_n} H_{n-1}^{\sing}(Y_{n-1})
  \xrightarrow{H_{n-1}^{\sing}(i_{n-1})} H^{\sing}_{n-1}(Y_{n-1},Y_{n-2})
\]
for $\partial_n$ the boundary operator of the pair $(Y_n,Y_{n-1})$ and the inclusion
$i_{n-1} \colon Y_{n-1} = (Y_{n-1},\emptyset) \to (Y_{n-1},Y_{n-2})$.

Given a pair of $\calf$-$G$-$CW$-complexes $(X,A)$, the filtration by its skeletons induces
filtrations on the spaces $X^H$ and $C_GH\backslash X^H$ for every $H \in \calf$.
We get a contravariant $\IZ\OrGF{G}{\calf}$-chain complex
$C_*^{\OrGF{G}{\calf}}(X,A) \colon \OrGF{G}{\calf} \to \Chcata{\IZ}$ and a contravariant
$\IZ\SubGF{G}{\calf}$-chain complex $C_*^{\SubGF{G}{\calf}}(X,A) \colon \SubGF{G}{\calf} \to \Chcata{\IZ}$
by putting
\begin{eqnarray*}
  C_*^{\OrGF{G}{\calf}}(X,A)(G/H)
  & := &
  C_*^c(O_X(G/H),O_A(G/H)) = C_*^c(X^H,A^H);
  \\
  C_*^{\SubGF{G}{\calf}}(X,A)(H)
  & := &
  C_*^c(S_X(X)(H),S_A(H)) = C_*^c(C_GH\backslash X^H,C_GH\backslash A^H).
\end{eqnarray*}

Choose a $G$-pushout
\begin{equation}
  \xymatrix@!C=11em{\coprod_{i\in I_n} G/H_i \times S^{n-1}
    \ar[r]^-{\coprod_{i\in I_n}  q^n_i} \ar[d]
&
X_{n-1} \ar[d]
\\
\coprod_{i \in I_n} G/H_i \times D^n \ar[r]_-{\coprod_{i\in I_n}  Q^n_i} 
&
X_{n-1}.
}
\label{pushout_for_(X_n,X_(n-1)}
\end{equation}
It induces for every closed subgroup $H \subseteq G$  pushouts
\[
  \xymatrix@!C=11em{\coprod_{i\in I_n} (G/H_i)^H \times S^{n-1}
    \ar[r]^-{\coprod_{i\in I_n}  (q^n_i)^H} \ar[d]
&
X_{n-1}^H \ar[d]
\\
\coprod_{i \in I_n} (G/H_i)^H \times D^n \ar[r]_-{\coprod_{i\in I_n}  (Q^n_i)^H} 
&
X_{n-1}^H
}
\]
and
\[
  \xymatrix@!C=15em{\coprod_{i\in I_n} C_GH\backslash (G/H_i)^H \times S^{n-1}
    \ar[r]^-{\coprod_{i\in I_n}  C_GH\backslash(q^n_i)^H} \ar[d]
&
C_GH\backslash X_{n-1}^H \ar[d]
\\
\coprod_{i \in I_n} C_GH\backslash (G/H_i)^H \times D^n
\ar[r]_-{\coprod_{i\in I_n}  C_GH\backslash(Q^n_i)^H} 
&
C_GH\backslash X_{n-1}^H.
}
\]
Note that $(G/H_i)^H$ agrees with $\mor_{\OrGF{G}{\calf}}(G/H,G/H_i) = \map_G(G/H,G/H_i)$.
In the sequel we  denote by $\IZ S$ for a set $S$ the free $\IZ$-module with the set $S$ as basis.
Since singular homology satisfies the disjoint union axiom, homotopy invariance and excision,
we obtain an isomorphism of contravariant $\IZ \OrGF{G}{\calf}$-modules
\begin{equation}
\bigoplus_{i \in I_n} \IZ\mor_{\OrGF{G}{\calf}}(?,G/H_i) \xrightarrow{\cong} C_n^{\OrGF{G}{\calf}}(X,A),
\label{computing_C_n_upper_Or(G)}
\end{equation}
where $\IZ\mor_{\OrGF{G}{\calf}}(?,G/H_i)$ is the free $\IZ\Or(G)$-module based at the object $G/H_i$,
see~\cite[Example~9.8 on page~164]{Lueck(1989)}, and analogously an isomorphism of  contravariant
$\IZ\SubGF{G}{\calf}$-modules
\begin{equation}
\bigoplus_{i \in I_n} \IZ\mor_{\IZ\SubGF{G}{\calf}}(?,H_i) \xrightarrow{\cong} C_n^{\SubGF{G}{\calf}}(X,A).
\label{computing_C_n_upper_Sub(G)}
\end{equation}

If $P_*C_*^{\OrGF{G}{\calf}}(X,A)$ is the $\IZ\SubGF{G}{\calf}$-chain complex obtained 
by induction with $P \colon \OrGF{G}{\calf} \to  \SubGF{G}{\calf}$ from $C_*^{\OrGF{G}{\calf}}(X,A)$,
see~\cite[Example~9.15 on page~166]{Lueck(1989)}, we conclude
from~\eqref{computing_C_n_upper_Or(G)} and~\eqref{computing_C_n_upper_Sub(G)}
that the canonical  map of $\IZ\SubGF{G}{\calf}$-chain complexes
\begin{equation}
 P_*C_*^{\OrGF{G}{\calf}}(X,A) \xrightarrow{\cong} C_*^{\SubGF{G}{\calf}}(X,A)
  \label{iso_P_ast_C_upper_oR_to_C_upper_Sub}
\end{equation}
is an isomorphism.

For a covariant $\IZ\Or(G)$-module $M$, we get from the
tensor product over $\OrGF{G}{\calf}$, see~\cite[9.13 on page~166]{Lueck(1989)},
a $\IZ$-chain complex $C_*^{\OrGF{G}{\calf}}(X,A)\otimes_{\IZ\OrGF{G}{\calf}} M$.

\begin{definition}[Bredon homology]\label{def:Bredon_homology}
 We define the $n$-th \emph{Bredon homology} to be the $\IZ$-module
\[
  B\!H^{G,\calf}_n(X,A;M) = H_n\bigl(C_*^{\OrGF{G}{\calf}}(X,A)\otimes_{\IZ\OrGF{G}{\calf}} M\bigr).
\]
Given a covariant $\IZ\SubGF{G}{\calf}$-module $N$, define analogously
\[
  S\!H^{G,\calf}_n(X,A;N) = H_n\bigl(C_*^{\SubGF{G}{\calf}}(X,A) \otimes_{\IZ\SubGF{G}{\calf}} N\bigr).
\]
\end{definition}

Given a covariant $\IZ\SubGF{G}{\calf}$-module $N$, define the covariant $\IZ\OrGF{G}{\calf}$-module
$P^*N$ to be $N \circ P$.  We get from the adjunction of~\cite[9.22 on
page~169]{Lueck(1989)} and~\eqref{iso_P_ast_C_upper_oR_to_C_upper_Sub} a natural
isomorphism of $\IZ$-chain complexes
\begin{equation}
  C_*^{\SubGF{G}{\calf}}(X,A) \otimes_{\IZ\SubGF{G}{\calf}} N
  \xrightarrow{\cong}  C_*^{\OrGF{G}{\calf}}(X,A) \otimes_{\IZ \OrGF{G}{\calf}} P^*N
  \label{Adjunction_between_chain_complexs_over_or_and_sub}
\end{equation}
and hence natural isomorphism of $\IZ$-modules
\begin{equation}
  B\!H^{G,\calf}_n(X,A;P^*N) \xrightarrow{\cong} S\!H^{G,\calf}_n(X,A;N).
  \label{Adjunction_between_BH_and_SH}
\end{equation}

Let $(X,A) $ be a pair of $\calf$-$CW$-complexes. Denote by $\cali$ the set of isotropy
groups of points in $X$.  Let $M$ be a covariant $\IZ\OrGF{G}{\calf}$-module and $N$ be a
covariant $\SubGF{G}{\calf}$-module.  Denote by $M|_{\cali}$ and $N|_{\cali}$ their
restrictions to $\OrGF{G}{\cali}$ and $\SubGF{G}{\cali}$.  Then one easily checks
using~\cite[Lemma~1.9]{Davis-Lueck(1998)}
that there are canonical isomorphisms
\begin{eqnarray}
  B\!H^{G,\cali}_n(X,A;M|_{\cali}) & \cong & B\!H^{G,\calf}_n(X,A;M);
\label{BH_and_enlarging_family}
  \\
  S\!H^{G,\cali}_n(X,A;N|_{\cali}) & \cong & B\!H^{G,\calf}_n(X,A;N).
\label{SH_and_enlarging_family}
\end{eqnarray}


\subsection{The construction of the equivariant Atiyah-Hirzebruch spectral sequence}%
\label{subsec:The_construction_the_equivariant_Atiyah-Hirzebruch_spectral_sequence_smooth}

\begin{proof}[Proof of Theorem~\ref{the:equivariant_Atyiah-Hirzebruch_spectral_sequence_smooth}]
   Since $(X,A)$ comes with the skeletal filtration, there is by a general construction a spectral sequence
   \[
   E_{p,q}^r,\quad   d^r_{p,q} : E^r_{p,q} \to E^r_{p-r,q+r-1}
   \]
converging to $\calh^G_{p+q}(X,A)$, whose $E_1$-term is given by
\[
  E^1_{p,q}  =  \calh_{p+q}^{G}(X_p,X_{p-1})
\]
and the first differential is the composite
\[d^1_{p,q} : E^1_{p,q} =  \calh_{p+q}^{G}(X_p,X_{p-1})
\to \calh_{p+q-1}^{G}(X_{p-1})
\to \calh_{p+q-1}^{G}(X_{p-1},X_{p-2})=  E^1_{p-1,q},
\]
where the first map is the boundary operator of the pair $(X_p,X_{p-1})$ and the second is
induced by the inclusion. The elementary construction is explained for trivial $G$ for
instance in~\cite[15.6 on page~339]{Switzer(1975)}. The construction
carries directly over to the equivariant setting.

The straightforward proof of the identification of $E^2_{p,q}$ with $B\!H^{G,\calf}_p(X,A;\calh_q)$
is left to the reader. 
\end{proof}

\subsection{Passing to the subgroup category}\label{subsec:Passing_to_the_subgroup_category}

\begin{condition}[Sub$|_{\calf}$]\label{con:(Sub)}
  Let $\calh^G_*(-)$ be a smooth $G$-homology theory.  
  Then $\calh^G_*(-)$ satisfies the Condition (Sub$|_{\calf}$) if for any $H \in \calf$ and
  $g \in C_GH$ the $G$-map $R_{g^{-1}} \colon G/H \to G/H$ sending $g'H$ to $g'g^{-1}H$
  induces the  identity on $\calh^G_q(G/H)$, i.e., $\calh^G_q(R_{g^{-1}}) = \id_{\calh^G_q(G/H)}$.
\end{condition}

\begin{remark}\label{rem:passage_to_Sub}
  Suppose that the $G$-homology theory $\calh^G_*$ satisfies the Condition (Sub$|_{\calf}$).
  Then the covariant $\IZ\OrGF{G}{\calf}$-module $\calh^G_q$
  sending $G/H$ with $H \in \calf$ to $\calh^G_q(G/H)$
  defines a covariant $\IZ\SubGF{G}{\calf}$-module
  $\overline{\calh^G_q} \colon \SubGF{G}{\calf}\to \MODcat{\IZ}$ uniquely determined by
  $\calh_q^G = \overline{\calh^G_q} \circ P$ for the projection
  $P \colon \OrGF{G}{\calf} \to \SubGF{G}{\calf}$. Moreover, we obtain
  from~\eqref{Adjunction_between_BH_and_SH} 
  for every pair $(X,A)$ of $\calf$-$G$-$CW$-complexes  natural isomorphisms
  \[
    B\!H^{G,\calf}_n(X,A;\calh_q^G(-)) \xrightarrow{\cong} S\!H^{G,\calf}_n(X,A;\overline{\calh_q^G(-)}).
  \]
  Note that the right hand side is often easier to compute than the left hand side.  One
  big advantage of $\Sub(G)$ in comparison with $\Or(G)$ is that for a finite subgroup
  $H \subseteq G$ the set of automorphisms of $H$ is the group $N_GH/H\cdot C_GH$, which
  is finite, whereas the set of automorphisms of $G/H$ in $\Or(G)$ for a finite group $H$
  is the group $N_GH/H$, which is not necessarily finite. This is a key ingredient in the
  construction of an equivariant Chern character for discrete groups $G$ and proper
  $G$-$CW$-complexes in~\cite{Lueck(2002b),Lueck(2002d)}.

  If $G$ is abelian, $\SubGF{G}{\calf}$ reduces to the poset of open subgroups of $G$ ordered by inclusion.
\end{remark}


\subsection{The connective case}%
\label{subsec:The_connective_case}

\begin{theorem}\label{the:connective_case}

\begin{enumerate}
\item\label{the:connective_case:q_smaller_0}
  Suppose that $\calh^G_q(G/H) = 0$ for every  $H \in \calf$ and $q \in \IZ$ with $q < 0$.
  Then we get for every  pair $(X,A)$  of  $\calf$-$G$-$CW$-complexes
  and every    $q \in \IZ$ with $q < 0$
  \[
  \calh^G_q(X,A) = 0;
  \]

\item\label{the:connective_case:q_is_0}
  Choose a model  $\EGF{G}{\COP}$ for the classifying space of smooth proper $G$-actions.
  Let $\cali$ be the set of isotropy groups of points in  $\EGF{G}{\COP}$.
  Suppose that $\calh^G_q(G/H) = 0$ for every open $H \in \cali$ and $q \in \IZ$ with $q < 0$.

  \begin{enumerate}
  \item\label{the:connective_case:q_is_0_Or}
    Then for every $q < 0$ we have $\calh^G_q(\EGF{G}{\COP}) = 0$,
    the edge homomorphism induces an isomorphism
    \[
    B\!H_0^G(\EGF{G}{\COP};\calh_q^G(-)) \xrightarrow{\cong} \calh^G_0(\EGF{G}{\calf})
    \]
 and the canonical map
 \[
 \colimunder_{G/H \in \OrGF{G}{\cali}} \calh^G_0(G/H) \xrightarrow{\cong} \calh^G_0(\EGF{G}{\calf})
\]
is bijective;

\item\label{the:connective_case:q_is_0_Sub}
  Suppose additionally  that $\calh^G_*$ satisfies Condition (Sub$_\cali$), see Condition~\ref{con:(Sub)}.
  Then the edge homomorphism induces an isomorphism
    \[
    S\!H_0^G(\EGF{G}{\COP};\overline{\calh_q^G(-)}) \xrightarrow{\cong} \calh^G_0(\EGF{G}{\calf})
  \]
  and  the canonical map
 \[
 \colimunder_{H\in \SubGF{G}{\cali}} \overline{\calh^G_0}(H) \xrightarrow{\cong} \calh^G_0(\EGF{G}{\COP})
\]
is bijective.
\end{enumerate}
\end{enumerate}
\end{theorem}
\begin{proof}~\ref{the:connective_case:q_smaller_0} This follows directly from the smooth equivariant
  Atyiah-Hirzebruch spectral sequence of Theorem~\ref{the:equivariant_Atyiah-Hirzebruch_spectral_sequence_smooth}
 \\[1mm]~\ref{the:connective_case:q_is_0_Or}
We get $\calh^G_q(\EGF{G}{\COP}) = 0$ for $q < 0$ from assertion~\ref{the:connective_case:q_smaller_0}.

We get from the the smooth equivariant
Atyiah-Hirzebruch spectral sequence of Theorem~\ref{the:equivariant_Atyiah-Hirzebruch_spectral_sequence_smooth}
an isomorphism
\[
  B\!H_0^{G,\cali}(\EGF{G}{\COP};\calh^G_0) = H_0(C_*^{\OrGF{G}{\cali}}(\EGF{G}{\COP}) \otimes_{\IZ\OrGF{G}{\cali}}\calh^G_0)
 \xrightarrow{\cong}  \calh_0^G(\EGF{G}{\COP}).
\]
since $E^2_{p,q} = B\!H_0^{G,\cali}(\EGF{G}{\COP};\calh^G_q) = 0$ is valid for $p,q \in \IZ$ if $p < 0$ or $q < 0$  holds.
Since  the $\IZ\OrGF{G}{\cali}$-module
$C_n^{\OrGF{G}{\cali}}(\EGF{G}{\COP})$ is free in the sense of~\cite[9.16 on page~167]{Lueck(1989)}
for $n \ge 0$ by~\eqref{computing_C_n_upper_Or(G)} and $\EGF{G}{\COP}^H$ is weakly contractible for $H \in \cali$,
the $\IZ \OrGF{G}{\cali}$-chain complex
$C_*^{\OrGF{G}{\cali}}(\EGF{G}{\COP})$ is a projective
$\IZ\OrGF{G}{\cali}$-resolution of the constant contravariant
$\IZ \OrGF{G}{\cali}$-module $\underline{\IZ}$, whose value is $\IZ$
at each object and assigns to any morphism $\id_{\IZ} $. Since $- \otimes_{\IZ \otimes_{\IZ\OrGF{G}{\cali}}}\calh^G_q$ is
right exact by~\cite[9.23 on page~169]{Lueck(1989)}, we get a 
isomorphism
\[
  H_0(C_*^{\OrGF{G}{\cali}}(\EGF{G}{\COP}) \otimes_{\IZ\OrGF{G}{\cali}}\calh^G_0)
  \cong 
  \underline{\IZ} \otimes_{\IZ\OrGF{G}{\cali}}\calh^G_0.
\]
We conclude from the adjunction appearing in~\cite[9.21 on page~169]{Lueck(1989)}
and the universal property of the  colimit  that there is a canonical isomorphism
\[
  \colimunder_{G/H \in \OrGF{G}{\cali}} \calh^G_0(G/H) 
  \cong
  \underline{\IZ} \otimes_{\IZ\OrGF{G}{\cali}}\calh^G_0.
\]
This finishes the proof of assertion~\ref{the:connective_case:q_is_0_Or}.
\\[1mm]~\ref{the:connective_case:q_is_0_Sub} This follows from
assertion~\ref{the:connective_case:q_is_0_Or}, since we get from Condition (Sub$_\cali$) a
canonical isomorphism
\[\colimunder_{G/H \in \OrGF{G}{\cali}} \calh^G_0(G/H) \xrightarrow{\cong} 
  \colimunder_{H\in \SubGF{G}{\cali}} \overline{\calh^G_q}(H).
\]
for the covariant $\IZ \SubGF{G}{\cali}$-module $\overline{\calh^G_q}$ determined by the
covariant $\IZ\OrGF{G}{\cali}$-module $\calh^G_q$, see Remark~\ref{rem:passage_to_Sub}.
\end{proof}


\subsection{The first differential}\label{subsec:The_first_differential}

Let $X$ be an $\calf$-$G$-$CW$-complex.  Suppose that $X_0 = \coprod_{j \in J} G/V_j$ and
that $X_1$ is given by the $G$-pushout
\begin{equation}
  \xymatrix@!C=11em{\coprod_{i\in I} G/U_i \times S^0
    \ar[r]^-{\coprod_{i\in I_n}  q_i} \ar[d]
&
X_0 \ar[d]
\\
\coprod_{i \in I} G/U_i \times D^1 \ar[r]_-{\coprod_{i\in I}  Q_i} 
&
X_1.
}
\label{pushout_for_(X_1,X_0}
\end{equation}
We want to figure out the map of $\IZ\OrGF{G}{\calf}$-modules $\gamma$ making the following
diagram commute
\[
  \xymatrix{\bigoplus_{i \in I} \IZ\mor_{\OrGF{G}{\calf}}(?,G/U_i) \ar[d]_{\cong} \ar[r]^-{\gamma}
    &
    \bigoplus_{j \in J} \IZ\mor_{\OrGF{G}{\calf}}(?,G/V_j) \ar[d]^{\cong} 
    \\
    C_1^{\OrGF{G}{\calf}}(X) \ar[r]_{c_n}
    &
    C_{0}^{\OrGF{G}{\calf}}(X)
  }
\]
where the vertical isomorphisms come from the
isomorphisms~\eqref{computing_C_n_upper_Or(G)}.
In order to describe $\gamma$, we have to define for each $i \in I$ and $j \in J$
a map of $\IZ\Or(G)$-modules
\[
\gamma_{i,j} \colon \IZ\mor_{\OrGF{G}{\calf}}(?,G/H_i) \to \IZ\mor_{\IZ\OrGF{G}{\calf}}(?,G/K_j)
\]
such that $\{j \in I_{n-1}\mid \gamma_{i,j} \not= 0\}$ is finite for every $i \in I_n$.
Note that $\gamma_{i,j}$ is determined by the image of $\id_{G/H_i}$. Hence we need
to specify for $i \in I$ and $j \in J$ an element
\begin{equation}
  \overline{\gamma_{i,j}} \in \IZ\mor_{\OrGF{G}{\calf}}(G/U_i,G/V_j) = \IZ\map_G(G/U_i,G/V_j).
  \label{overline(gamma_(i,j))}
\end{equation}

For each $i \in I$ there are two elements $j_{-}(i)$ and $j_{+}(i)$ in $J$ such that the
image of $G/H_i \times \{\pm 1\}$ under the map $q_i$ appearing
in~\eqref{pushout_for_(X_1,X_0} is the summand $G/K_{j_{\pm }}(i)$ belonging to
$j_{\pm}(i)$ of $\coprod_{j \in I_0} G/K_j$, if we write $S^0 = \{-1,1\}$.  Denote by
$(q_i^1)_{\pm 1} \colon G/H_i \to G/K_{j_{\pm}}$ the restriction of $q_i^1$ to
$G/H_i \times \{\pm 1\}$.  We leave the elementary proof of the next lemma to the reader.

    \begin{lemma}\label{lem:value_of_overline(gamma_(i,j)_n_is_1} 
      We get in $\IZ\map_G(G/H_i,G/K_j)$
      \[
        \overline{\gamma_{i,j}}  =
        \begin{cases}
          \pm [(q_i^1)_{\pm 1}]
          &
          \text{if} \; j = j_{\pm}(i) \; \text{and}\; j_{-}(i) \not= j_{+}(i);
         \\
         [(q_i^1)_{+1}] - [(q_i^1)_{-1}]
         &
         \text{if} \; j = j_{-}(i) = j_{+}(i);
         \\
         0 & \text{if}\; j \notin \{ j_{-}(i), j_{+}(i)\}.
       \end{cases}
       \]
     \end{lemma}

\begin{remark}\label{rem:computation_of_bredon_chain-complex}
     This implies for the $\IZ$-chain complex
     $C_*^{\OrGF{G}{\calf}}(X,A) \otimes_{\OrGF{G}{\calf}} M$ for a covariant $\IZ\OrGF{G}{\calf}$-module $M$ that its first
     differential agrees with  the $\IZ$-homomorphism
     \[
     \alpha = (\alpha_{i,j})_{i\in I, j \in J} \colon \bigoplus_{i \in I} M(G/U_i) \to \bigoplus_{j \in J} M(G/V_j),
     \]
     where the $\IZ$-homomorphisms $\alpha_{i,j} \colon M(G/U_i) \to M(G/V_j)$ are
     given as follows.  We get in the notation of
     Lemma~\ref{lem:value_of_overline(gamma_(i,j)_n_is_1}
       \[
        \alpha_{i,j}  =
        \begin{cases}
          \pm M((q_i^1)_{\pm})
          &
          \text{if} \; j = j_{\pm}(i) \; \text{and}\; j_{-}(i) \not= j_{+}(i);
         \\
         M((q_i^1)_{+1}) - M((q_i^1)_{-1})
         &
         \text{if} \; j = j_{-}(i) = j_{+}(i);
         \\
         0 & \text{if}\; j \notin \{ j_{-}(i), j_{+}(i)\}.
       \end{cases}
     \]
      Note  that the cokernel of $\alpha$ is $B\!H_0^{G,\calf}(X;M)$.

      We get a computation of the first differential of
      $C_*^{\SubGF{G}{\calf}}(X,A) \otimes_{\IZ\SubGF{G}{\calf}} N$ for a covariant $\IZ\SubGF{G}{\calf}$-module $N$
       from the isomorphism~\eqref{Adjunction_between_chain_complexs_over_or_and_sub}.
     Explicitly the first differential is given by
     \[
     \beta = (\beta _{i,j})_{i\in I, j \in J} \colon \bigoplus_{i \in I_n} N(U_i) \to \bigoplus_{j \in I_{n-1}} N(V_j),
     \]
     where the $\IZ$-homomorphisms $\beta_{i,j} \colon N(G/U_i) \to N(G/V_j)$ are
     given as follows. Choose for the map $(q_i)_{\pm} \colon G/U_i \to G/V_j$ an element $(g_i)_{\pm}$ with
     $(q_i)_{\pm}(eU_i) = (g_i)_{\pm}^{-1}V_j$. Let $[c(g_i)_{\pm}] \colon U_i \to V_j$ be the morphism in
     $\SubGF{G}{\calf}$ represented by $c(g_i)_{\pm} \colon U_i \to V_j$ sending $u$ to $gug^{-1}$. Then
       \[
        \beta_{i,j}  =
        \begin{cases}
          \pm N([c(g_i)_{\pm}])
          &
          \text{if} \; j = j_{\pm}(i) \; \text{and}\; j_{-}(i) \not= j_{+}(i);
         \\
         N([c(g_i)_{+}]) - N([c(g_i)_{-}])
         &
         \text{if} \; j = j_{-}(i) = j_{+}(i);
         \\
         0 & \text{if}\; j \notin \{ j_{-}(i), j_{+}(i)\}.
       \end{cases}
     \]
      Note  that the cokernel of $\beta $ is $S\!H_0^{G,\calf}(X;N)$.       
     \end{remark}


\section{A brief review of the Farrell Jones Conjecture for the algebraic $K$-theory of Hecke algebras}%
\label{sec:A_brief_review_of_the_farrell_Jones_Conjecture_for_the_algebraic_K-theory_of_Hecke_categories}

\typeout{------ Section 3: A brief review of the Farrell Jones Conjecture for the algebraic $K$-theory of Hecke algebras -----}

In this section we give a review of the Farrell Jones Conjecture for the algebraic
$K$-theory of Heckes algebras.
Further information can be found in~\cite{Bartels-Lueck(2022KtheoryHecke),
  {Bartels-Lueck(2023K-theory_red_p-adic_groups)}}.

Let $R$ be a (not necessarily commutative) associative unital ring with $\IQ \subseteq
R$. Let $G$ be a td-group.  Let $\calh(G;R)$ be the associated Hecke algebra.

 One can construct a covariant functor
\[
  \bfK_{R} \colon \OrGF{G}{\OPEN}  \to  \Spectra;
 \]
such that $\pi_n(\bfK_{R}(Q'/U')) \cong   K_n(\calh(U;R))$ 
holds for any $n \in \IZ$ and open subgroup $U \subseteq Q$.
Associated to it is a smooth $G$-homology theory $H^G_*(-;\bfK_R)$
such that
\begin{equation}
  H^G_n(G/U;\bfK_R) \cong   K_n(\calh(U;R))
  \label{H_upper_G_n(G/U;bfK_R)_cong_K_n(calh(U;R))}
\end{equation}
holds for every $n \in \IZ$ and every open subgroup $U \subseteq Q$.

The next result follows from~\cite[Corollary~1.8]{Bartels-Lueck(2023K-theory_red_p-adic_groups)}.

\begin{theorem}\label{the:FJC_for_Hecke_algebras}
  Let $G$ be a td-group which is modulo a normal compact subgroup a subgroup
  of a reductive $p$-adic group.
  Let $R$ be a uniformly regular ring with $\IQ \subseteq R$.

  Then the map induced by the projection $\EGF{G}{\COP} \to G/G$ induces
  for every $n \in \IZ$ an isomorphism
  \[
  H^G_n(\EGF{G}{\COP};\bfK_R) \xrightarrow{\cong} H^G_n(G/G;\bfK_R) = K_n(\calh(G;R)).
  \]
\end{theorem}


\section{Proof of the Main Theorem~\ref{the:Main_Theorem}}%
\label{sec:Proof_of_the_Main_Theorem_ref(the:Main_Theorem)}

\typeout{-------------------- Section 4: Proof of the Main Theorem~\ref{the:Main_Theorem} --------------------}
 
\begin{proof}[Proof of Theorem~\ref{the:Main_Theorem}]~\ref{the:Main_Theorem:assembly} 
 This is exactly Theorem~\ref{the:FJC_for_Hecke_algebras}. 
\\[1mm]~\ref{the:Main_Theorem:spectral_sequence}
Since an open group homomorphism $U \to V$ between two td-groups induces a ring
 homomorphism $\calh(U;R) \to \calh(V;R)$ between the Hecke algebras and hence a
 homomorphism $K_n(\calh(U;R)) \to K_n(\calh(V;R))$ and inner automorphisms of a td-group
 $U$ induce the identity on $K_n(\calh(U;R))$, we get a covariant
 $\IZ\SubGF{G}{\COM}$-module $K_n(\calh(?;R))$ whose value at $U$ is
 $K_n(\calh(U;R))$. Since the
 isomorphism~\eqref{H_upper_G_n(G/U;bfK_R)_cong_K_n(calh(U;R))} is natural, we get an
 isomorphisms of covariant $\IZ\OrGF{G}{\OPEN}$-modules
 \[
   P^*K_n(\calh(?;R)) \xrightarrow{\cong} \pi_n(\bfK_{R})
 \]
 for the projection $P \colon \OrGF{G}{\OPEN} \to \SubGF{G}{\OPEN}$
 of~\eqref{projection_Or(G)_to_Sub(G)}.  So the smooth equivariant Atiyah-Hirzebruch
 spectral sequence applied to the smooth homology theory $H_*^G(-;\bfK_{R})$ takes for a
 $\calf$-$G$-$CW$-complexes $X$ the form
 \begin{equation}
   E_{p,q}^2 = S\!H^{G,\calf}_q\bigl(X;K_q(\calh(?;R))\bigr)
   \implies H^G_{p+q}(X;\bfK_{R}).
   \label{equivariant_Atiyah-Hirzebruch_spectral_sequence_for_bfK_R}
 \end{equation}
 Now assertion~\ref{the:Main_Theorem:spectral_sequence} follows from the special case
 $X = \EGF{G}{\COP}$ and assertion~\ref{the:Main_Theorem:assembly}.
 \\[1mm]~\ref{the:Main_Theorem:K_0} and~\ref{the:Main_Theorem:negative}
 As $K_q(\calh(K;R))$ vanishes for every compact td-group $K$ and every $q \le -1$,
 see~\cite[Lemma~8.1]{Bartels-Lueck(2022KtheoryHecke)},
 assertions~\ref{the:Main_Theorem:K_0}, and~\ref{the:Main_Theorem:negative}
 follow from Theorem~\ref{the:connective_case} applied in
 the case $X = \EGF{G}{\COP}$ and from assertion~\ref{the:Main_Theorem:assembly}. This
 finishes the proof of the Main Theorem~\ref{the:Main_Theorem}.
\end{proof}


\section{The main recipe for the computation of the projective class group}%
\label{sec:The_main_recipe_for_the_computation_of_the_projective_class_group}

\typeout{-- Section 5 : The main recipe for the computation of the projective class group  --}

Throughout this section $G$ will be a td-group and $R$ be a uniformly regular
ring with $\IQ \subseteq R$, e.g., a field of characteristic zero.
We will assume that the assembly map
$H_n^G(\EGF{G}{\COP};\bfK_R) \to H_n^G(G/G;\bfK_R) = K_n(\calh(G;R))$ is bijective for all $n \in \IZ$
This is known to be true for subgroups of reductive $p$-adic groups
by Theorem~\ref{the:FJC_for_Hecke_algebras}.


\subsection{The general case}\label{subsec:The_general_case}

Let $X$ be an abstract simplicial complex with a simplicial $G$-action such that all
isotropy groups are compact open, the $G$-action is cellular, and
$|X|^K$ is non-empty and connected  for every compact open subgroup $K$ of $G$.
 
We can choose a subset $V$ of the set of vertices of $X$
such that the $G$-orbit
through any vertex in $X$ meets $V$ in precisely one element.  Fix a total ordering on $V$. Let $E$ be
the subset of $V \times V$ consisting  of those pairs $(v,w)$ such that  $v \le   w$ holds
and  there exists
$g \in G$ for which $v$ and $gw$ satisfy $v \not= gw$ and span an edge $[v,gw]$ in $X$.  For $(v,w) \in E$ define
$\overline{F(v,w)}$ to be the subset of $G_{v}\backslash G /G_{w}$ consisting of elements
$x$ for which $v$ and $gw$ satisfy $v \not= gw$ and span an edge $[v,gw]$ in $X$ for
some (and hence all) representative $g$ of $x$. Choose a subset $F(v,w)$ of $G$ such that
the projection $G \to G_{v}\backslash G /G_{w}$ induces a bijection
$F(v,w) \to \overline{F(v,w)}$.

Then for every edge of $X$ the $G$-orbit through it meets the set
$\{[v, gv] \mid (v,w) \in E, g \in F(v,w)\}$ in precisely one element. Moreover, the
$0$-skeleton of $|X|$ is given by $|X|_0 = \coprod_{u \in V} G/G_{u}$ and $|X|_1$ is
given by the $G$-pushout
\[
  \xymatrix@!C=19em{\coprod_{(v,w) \in E} \coprod_{g \in F(v,w)} G/(G_{v} \cap G_{gw})\times S^0
    \ar[r]^-{\coprod_{(v,w) \in E} \coprod_{g \in F(v,w)} q_{(v,w),g}} \ar[d]
    &
    |X|_0 \ar[d]
    \\
    \coprod_{(v,w) \in E} \coprod_{g \in F(v,w)} G/(G_{v} \cap G_{gw}) \times D^1
    \ar[r]
    &
    |X|_1
  }
\]
where
$q_{(v,w),g} \colon G/(G_{v} \cap G_{gw}) \times S^0 \to |X|_0 = \coprod_{u \in V}G/G_{u}$
is defined as follows. Write $S^0 = \{-1,1\}$. The restriction of $q_{(v,w),g}$
to $G/(G_{v} \cap G_{gw}) \times \{-1\}$ lands in the summand $G/G_{v}$ and is given
by canoncial projection.  The restriction of $q_{(v,w),g}$ to
$G/(G_{v} \cap G_{gw}) \times \{1\}$ lands in the summand $G/G_{w}$ and is given by the $G$-map
$R_{g^{-1}} \colon G/(G_{v} \cap G_{gw}) \to G/G_{w}$ sending
$z(G_{v} \cap G_{gw})$ to $zgG_w$.

Next we define a map
\[
\beta = (\beta_{(v,w),g,u}) \colon \bigoplus_{(v,w) \in E} \bigoplus_{g \in F(v,w)} K_0(\calh(G_v \cap G_{gw};R))
\to \bigoplus_{u \in V} K_0(\calh(G_{u};R)).
\]
If $u = v $, then
$\beta_{(v,w),g,v} \colon K_0(\calh(G_{v} \cap G_{gw};R)) \to K_0(\calh(G_{v};R))$
is the map induced by the inclusion $G_{v} \cap G_{gw} \to G_{v}$ multiplied
with $(-1)$.  If $u  = w$, then $\beta_{(v,w),g,w}$ 
$K_0(\calh(G_v\cap G_{gw};R))\to K_0(\calh(G_{w};R))$ is the map induced by the group
homomorphism $ G_v \cap G_{gw} \to G_{w}$ sending $z$ to $g^{-1}zg$.  If
$u \notin \{v,w\}$, then $\beta_{(v,w),g,u} = 0$. 

\begin{lemma}\label{lem:cokernel_of_beta_is_K_0(calh(G;R)}
  The cokernel of $\beta$ is isomorphic to $K_0(\calh(G;R))$.
\end{lemma}
\begin{proof}
  We conclude from Remark~\ref{rem:computation_of_bredon_chain-complex}  that the
  cokernel of $\beta$ is $S\!H_0^{G,\COP}(X;\overline{K_0^G(-)})$.  The up to $G$-homotopy
  unique $G$-map $f \colon X \to \EGF{G}{\COP}$ induces for every compact open subgroup
  $K \subset G$ a $1$-connected map $f^{K} \colon |X|^K \to \EGF{G}{\COP}^K$. This implies
  that the map
  $S\!H_0^{G,\COP}(X;\overline{K_0^G(-)}) \to
  S\!H_0^{G;\COP}(\EGF{G}{\COP};\overline{K_0^G(-)})$ induced by $f$ is an isomorphism,
  see~\cite[Proposition~23~(iii) on page~35]{Lueck(1989)}.
  Theorem~\ref{the:connective_case}~\ref{the:connective_case:q_is_0_Sub} implies
  $S\!H_0^G(\EGF{G}{\COP};\overline{K_0^G(-)}) \cong H_0^G(\EGF{G}{\COP};\bfK_R)$.
  Since by  
  assumption we have $H_0^G(\EGF{G}{\COP};\bfK_R) \cong K_0(\calh(G;R))$,
  Lemma~\ref{lem:cokernel_of_beta_is_K_0(calh(G;R)} follows.
\end{proof}

\begin{remark}\label{rem:lem:cokernel_of_beta_is_K_0(calh(G;R)_strict_fundamental_domain}
  Suppose additionally that $X$ possesses a strict fundamental domain $\Delta$, i.e., a
  simplicial subcomplex $\Delta$ that contains exactly one simplex from each orbit for the
  $G$-action on the set of simplices of $X$.  Then one can take $V$ to be the set of
  vertices of $\Delta$ and for $(v,w) \in E$ the set $F(v,w)$ to be $\{e\}$. Moreover,
  $\beta$ reduces to the map
  \[
    \beta = (\beta_{(v,w,u)}) \colon \bigoplus_{(v,w) \in E} K_0(\calh(G_v \cap G_{w};R))
    \to \bigoplus_{u \in V} K_0(\calh(G_{u};R)).
  \]
  where $\beta_{(v,w),u}$ is the map induced by the inclusion $G_{v} \cap G_{w} \to G_{v}$
  multiplied with $(-1)$ for $u = v$, the map induced by the inclusion
  $G_{v} \cap G_{w} \to G_{w}$ for $u = w$, and zero for $u \notin \{v,w\}$. Note that $E$
  is the subset  of $V \times V$ consisting of elements $(v,w)$ for which $v < w$ holds
  and $v$ and $w$ span an edge $[v,w]$ in $\Delta$.
\end{remark}


\subsection{A variation}\label{subsec:A_variation}

Consider a central extension
$1 \to \widetilde{C} \to \widetilde{G} \xrightarrow{\pr} G \to 1$ of td-groups together
with a group homomorphism $\mu \colon \widetilde{G} \to \IZ$ 
such that
$\widetilde{C} \cap \widetilde{M}$ is compact for $\widetilde{M}: = \ker(\mu)$.  We still
consider the abstract simplicial complex $X$  of Subsection~\ref{subsec:The_general_case}
coming with a simplicial $G$-action such that all isotropy groups are compact open, and
$|X|^K$ is non-empty  and connected for every compact open subgroup $K$ of $G$. Furthermore, we will
assume that the assembly map
$H_n^{\widetilde{G}}(\EGF{\widetilde{G}}{\COP};\bfK_R) \to
H_n^{\widetilde{G}}(\widetilde{G}/\widetilde{G};\bfK_R) = K_0(\calh(\widetilde{G};R))$ is
bijective for all $n \in \IZ$.

If $\widetilde{C}$ is compact, then we can consider $X$ as a $\widetilde{G}$-$CW$-complex
by restricting the $G$-action with $\pr$ and Subsection~\ref{subsec:The_general_case}
applies. Hence we will assume that $\widetilde{C}$ is not compact, or, equivalently, that
$\widetilde{C}$ is not contained in the kernel $\widetilde{M}: = \ker(\mu)$.  Then the
index $m := [\IZ : \mu(C)]$ is a natural number $m \ge 1$.  We fix an element
$\widetilde{c} \in \widetilde{C}$ with $\mu(\widetilde{c}) = m$.  In sequel we choose for
every $g \in G$ an element $\widetilde{g}$ in $\widetilde{G}$ satisfying
$\pr(\widetilde{g}) = g$ and denote for an open subgroup $U \subseteq G$ by
$\widetilde{U} \subseteq \widetilde{G}$ its preimage under
$\pr \colon \widetilde{G} \to G$.  Let
\[
\gamma \colon 
      \bigoplus_{(v,w) \in E}  \bigoplus_{g \in F(v,w)}
      K_0(\calh(\widetilde{G_{v}} \cap \widetilde{G_{gw}} \cap \widetilde{M};R))
      \to      \bigoplus_{u \in V} K_0(\calh(\widetilde{G_{u}}\cap \widetilde{M};R))
 \]
be the map whose component for $(v,w) \in E$, $g \in F(v,w)$, and $u \in V$ 
is the  map 
\begin{equation}
    \gamma_{(v,w),g,u}\colon  K_0(\calh(\widetilde{G_{v}} \cap \widetilde{G_{gw}} \cap \widetilde{M};R))
    \to  K_0(\calh(\widetilde{G_{u}}\cap \widetilde{M};R))
    \label{gamma_((v,w),g,u)}
  \end{equation}
  defined next. If $u = v$, it is the map
  coming from the inclusion
  $\widetilde{G_{v}} \cap \widetilde{G_{gw}} \cap \widetilde{M}
  \to \widetilde{G_{v}} \cap \widetilde{M}$
  multiplied with $(-1)$.
  If $u = w$, it  is the map coming from the group homomorphism
  $\widetilde{G_{v}} \cap \widetilde{G_{gw}} \cap \widetilde{M} \to \widetilde{G_{w}} \cap \widetilde{M}$
  sending $x$ to $\widetilde{g} x\widetilde{g}^{-1}$.
  If $u \not \in \{v,w\}$, it  is trivial. Note that this definition  is independent of the choice of
  $\widetilde{g} \in \widetilde{G}$ satisfying $\pr(\widetilde{g}) = g$ for $g \in F(v,w)$.

\begin{lemma}\label{lem:cokernel_of_beta_is_K_0(calh(widetilde(G);R)}
  The cokernel of $\gamma$ is $K_0(\calh(\widetilde{G};R))$.
\end{lemma}

\begin{proof}
  Note that $|X| \times \IR$ carries the $G \times \IZ$-$CW$-complex structure coming
  from the product of the $G$-$CW$-complex structure on $|X|$ and the standard free
  $\IZ$-$CW$-structure on $\IR$.  Since the $\IZ$-CW-complex $\IR$ has precisely one
  equivariant $1$-cell and one equivariant $0$-cell, the set of equivariant $0$-cells of
  the $G \times \IZ$-$CW$-complex $|X| \times \IR$ can be identified with the set
  $V$ and the set of equivariant  $1$-cells can be identified with the
  disjoint union of $V$ and the
  set $\coprod_{(v,w) \in E} F(v,w)$.  Now the $0$-skeleton of $|X| \times \IR$ is
  given by the disjoint union
  $\coprod_{u \in V} \widetilde{G}/\widetilde{G_{u}}\times \IZ$ and the $1$-skeleton
  of $|X| \times \IR$ is given by the $G \times \IZ$-pushout           
\begin{equation}
  \xymatrix@!C=19em{\mbox{$\begin{array}{c}
    \coprod_{v \in V} \widetilde{G}/\widetilde{G_{v}} \times \IZ \times S^0
    \\
     \coprod
     \\
  \coprod_{(v,w) \in E} \coprod_{g \in F(v,w)} \widetilde{G}/(\widetilde{G_{v}} \cap \widetilde{G_{gw}})
  \times \IZ \times S^0
  \end{array}$}     
  \ar[r]^-{\widetilde{q}} \ar[d]
  &
  \coprod_{u \in V} \widetilde{G}/\widetilde{G_{u}} \times \IZ \ar[d]
  \\
  \mbox{$\begin{array}{c}
    \coprod_{v \in V}\widetilde{G}/\widetilde{G_{v}} \times \IZ \times D^1
    \\
     \coprod
     \\
  \coprod_{(v,w) \in E} \coprod_{g \in F(v,w)} \widetilde{G}/(\widetilde{G_{v}} \cap \widetilde{G_{gw}})
  \times \IZ \times S^0
  \end{array}$}     
  \ar[r]
  &
  (|X| \times \IR)_1
  }
  \label{G_times_Z-pushout}
\end{equation}
where $\widetilde{q}$ is given as follows. Write $S^0 = \{-1,1\}$.  Fix
$u \in V$. The restriction of $\widetilde{q}$ to the summand
$\widetilde{G}/\widetilde{G_{v}} \times \IZ \times \{\epsilon\}$ lands in the summand
$\widetilde{G}/\widetilde{G_{v}} \times \IZ$ and is given by $\id$ for $\epsilon = -1$
and by $\id \times \sh_1$ for $\epsilon = 1$, where $\sh_a \colon \IZ \to \IZ$ sends $b$
to $a+b$ for $a,b \in \IZ$.  Fix $(v,w) \in E$ and $g \in F(v,w)$.  The restriction of
$\widetilde{q}$ to the summand
$\widetilde{G}/(\widetilde{G_{v}} \cap \widetilde{G_{gw}}) \times \IZ \times \{-1\}$
belonging to $(v,w)$ and $g$ lands in the summand for $u = v$ and is the canonical
projection
$\widetilde{G}/(\widetilde{G_{v}} \cap \widetilde{G_{gw}}) \times \IZ \to
\widetilde{G}/\widetilde{G_{v}} \times \IZ$.  The restriction of $\widetilde{q}$ to the
summand
$\widetilde{G}/(\widetilde{G_{v}} \cap \widetilde{G_{gw}}) \times \IZ \times \{1\}$
belonging to $(v,w)$ and $g$ lands in the summand for $u = w$ and is the map
$R_{\widetilde{g}^{-1}} \times \id_{\IZ} \colon \widetilde{G}/(\widetilde{G_{v}} \cap
\widetilde{G_{gw}}) \times \IZ \to \widetilde{G}/\widetilde{G_{w}} \times \IZ$, where
$R_{\widetilde{g}^{-1}}$ sends
$\widetilde{z}(\widetilde{G_{v}} \cap \widetilde{G_{gw}})$ to
$\widetilde{z}\widetilde{g}^{-1}\widetilde{G_{w}}$.

We have the  group homomorphism
\[
  \iota := \pr \times \mu \colon \widetilde{G} \to G \times \IZ.
\]
Its kernel is $\widetilde{C} \cap \widetilde{M}$. Its image has finite index in $G \times \IZ$,
which agrees with the index $m$ of the image of $\mu$ in $\IZ$. 

We are interested in the $\widetilde{G}$-$CW$-complex $\iota^* (|X| \times \IR)$
obtained by restriction with $\iota$ from the $G \times \IZ$-$CW$-complex
$|X| \times \IR$.  
So we have to analyze how the $G \times \IZ$-cells in
$\iota^* (|X| \times \IR)$ viewed as $\widetilde{G}$-spaces decompose  as  disjoint union
of $\widetilde{G}$-cells.  Consider any open subgroup $U \subseteq G$. Then we obtain a
$\widetilde{G}$-homeomorphism
\[\alpha(U) \colon \coprod_{p = 0}^{m-1} \widetilde{G}/(\widetilde{U} \cap \widetilde{M})
  \xrightarrow{\cong} \iota^*\bigl(G/U \times \IZ\bigr)
\]
by sending the element $\widetilde{z}(\widetilde{U} \cap \widetilde{M})$ in the $p$-th
summand to  $(\pr(\widetilde{z})U, \mu(\widetilde{z}) + p)$. Next we have to analyze
the naturality properties of $\alpha(U)$. The following diagram commutes for $a \in \IZ$
  \[
   \xymatrix@!C=12em{\coprod_{p = 0}^{m-1} \widetilde{G}/(\widetilde{U} \cap \widetilde{M})
    \ar[r]^-{\alpha(U)} \ar[d]_{\widehat{\pi}^a}
    &
    \iota^*\bigl(G/U \times \IZ\bigr) \ar[d]^{\id \times \sh_a} \\
    \coprod_{p = 0}^{m-1} \widetilde{G}/(\widetilde{U} \cap \widetilde{M})
    \ar[r]_-{\alpha(U)}
    &
    \iota^*\bigl(G/U \times \IZ\bigr)
  }
  \]
  where $\widehat{\pi}$ sends the summand for $p = 0, \ldots, m-2$ by the identity to the
  summand for $p+1$ and sends the summand for $p = m-1$ to the summand for $p = 0$ by the
  map  $R_{\widetilde{c}} \colon \widetilde{G}/(\widetilde{U} \cap \widetilde{M}) \to
  \widetilde{G}/(\widetilde{U} \cap \widetilde{M})$ for $\widetilde{c} \in \widetilde{C}$
  satisfying $\mu(\widetilde{c}) = m$.  Note for the sequel
  that the endomorphism
  $\pi_n(\bfK_R(R_{\widetilde{c}}))$
  of $\pi_n(\bfK_R(\widetilde{G}/\widetilde{U} \cap \widetilde{M}))
  = K_0(\calh(\widetilde{U} \cap \widetilde{M}))$ is the identity,
  since conjugation with $\widetilde{c}$ induces the identity on $\widetilde{U} \cap \widetilde{M}$.

Consider two open subgroups $U$ and $V$ of $G$
and an element $g \in G$ with $gUg^{-1} \subseteq V$.  Then we  get well-defined
$\widetilde{G}$-maps
$R_{\widetilde{g}^{-1}} \colon \widetilde{G}/(\widetilde{U} \cap \widetilde{M}) \to
\widetilde{G}/(\widetilde{V} \cap \widetilde{M})$ sending
$\widetilde{z}(\widetilde{U} \cap \widetilde{M})$ to
$\widetilde{z}\widetilde{g}^{-1}(\widetilde{V} \cap \widetilde{M})$ and
$R_{g^{-1}} \times \id \colon \iota^*\bigl(G/U \times \IZ\bigr)
\to \iota^*\bigl(G/V \times \IZ\bigr)$ sending
$(zU, n)$ to $(zg^{-1}V, n)$
and the following diagram commutes
\[
  \xymatrix@!C=12em{\coprod_{p = 0}^{m-1} \widetilde{G}/(\widetilde{U} \cap \widetilde{M})
    \ar[r]^-{\alpha(U)} _{\cong}\ar[d]_{\coprod_{p = 0}^{m-1} R_{\widetilde{g}^{-1}}}
    &
    \iota^*\bigl(G/U \times \IZ\bigr) \ar[d]^{R_{g^{-1}} \times \sh_{\mu(\widetilde{g}^{-1})}}
    \\
    \coprod_{p = 0}^{m-1} \widetilde{G}/(\widetilde{V} \cap \widetilde{M})
    \ar[r]_-{\alpha(V)}^{\cong}
    &
    \iota^*\bigl(G/V \times \IZ\bigr).
  }
  \]
  In particular the following diagram commutes
  \[
  \xymatrix@!C=12em{\coprod_{p = 0}^{m-1} \widetilde{G}/(\widetilde{U} \cap \widetilde{M})
    \ar[r]^-{\alpha(U)} _{\cong}\ar[d]_{\pi^{\mu(\widetilde{g})}\circ \bigl(\coprod_{p = 0}^{m-1} R_{\widetilde{g}^{-1}}\bigr)}
    &
    \iota^*\bigl(G/U \times \IZ\bigr) \ar[d]^{R_{g^{-1}} \times \id}
    \\
    \coprod_{p = 0}^{m-1} \widetilde{G}/(\widetilde{V} \cap \widetilde{M})
    \ar[r]_-{\alpha(V)}^{\cong}
    &
    \iota^*\bigl(G/V \times \IZ\bigr).
  }
  \]

  Now we obtain from the $G \times \IZ$-pushout~\eqref{G_times_Z-pushout}  by applying restriction with $\iota$
  and the maps $\alpha_U$ above a $\widetilde{G}$-pushout describing how the $1$-skeleton
  of the $\widetilde{G}$-$CW$-complex $\iota^* (|X| \times \IR)$ is obtained from its
  $0$-skeleton and explicite descriptions of the attaching maps.

  In the sequel $A^m$ stands for the $m$-fold direct sum of copies of $A$ for an abelian
  group $A$ and $\pi \colon A^m \to A^m$ denotes the permutation map sending
  $(a_1, a_2, \ldots, a_m)$ to $(a_m,a_1, \ldots, a_{m-1})$ and $\operatorname{aug}  \colon A^m \to A$
  denotes the augmentation map sending $(a_1, \ldots, a_m)$ to $a_1 + \cdots + a_m$.
  
      Let $\delta$ be the map given by the direct sum
      \[\delta = \bigoplus_{v \in V} \delta_v\colon
        \bigoplus_{v \in V} K_0(\calh(\widetilde{G_{v}} \cap \widetilde{M};R))^m
      \to
      \bigoplus_{v \in V} K_0(\calh(\widetilde{G_{v}}\cap \widetilde{M};R))^m
       \]
    where $\delta_v \colon K_0(\calh(\widetilde{G_{v}} \cap \widetilde{M};R))^m
    \to K_0(\calh(\widetilde{G_{v}} \cap \widetilde{M};R))^m$
    is $\pi -\id$. Let
    \[
    \epsilon \colon \bigoplus_{(v,w) \in E}  \bigoplus_{g \in F(v,w)}
    K_0(\calh(\widetilde{G_{v}} \cap \widetilde{G_{gw}} \cap \widetilde{M};R))^m
    \to
    \bigoplus_{u \in V} K_0(\calh(\widetilde{G_{u}}\cap \widetilde{M};R))^m
  \]
  be the map given by the components $\epsilon_{(v,w),g,u}$ defined as follows.
  For $u = v$ the map $\epsilon_{(v,w),g,v}$  is the $m$-fold direct sum $\gamma_{(v,w),g,v}^m$ of the maps
  $\gamma_{(v,w),g,v}$ defined in~\eqref{gamma_((v,w),g,u)}.
   For $u = w$ we put
  \begin{multline*}
    \epsilon_{(v,w),g,w}   \colon
    K_0(\calh(\widetilde{G_{v}} \cap \widetilde{G_{gw}} \cap \widetilde{M};R))^m
    \xrightarrow{\gamma_{(v,w),g,u}^m}
    K_0(\calh(\widetilde{G_{v}}\cap \widetilde{M};R))^m
    \\
    \xrightarrow{\pi^{\mu(\widetilde{g})}} K_0(\calh(\widetilde{G_{w}}\cap \widetilde{M};R))^m.
  \end{multline*}
  Since $\pi^m = \id$,  the map $\pi^{\mu(\widetilde{g})}$ depends only on $\overline{\mu}(g)$,
   where $\overline{\mu} \colon G \to \IZ/m$ sends $g$ to the image of $\widetilde{g}$ under the projection
  $\IZ \to \IZ/m$ for any choice of an element $\widetilde{g} \in \widetilde{G}$ with $\pr(\widetilde{g}) = g$.

  The cokernel of the map
  \begin{multline*}
  \delta \oplus \epsilon\colon 
  \left(\bigoplus_{v \in V} K_0(\calh(\widetilde{G_{v}} \cap \widetilde{M};R))^m\right)
    \oplus
    \left(\bigoplus_{(v,w) \in E}  \bigoplus_{g \in F(v,w)}
      K_0(\calh(\widetilde{G_{v}} \cap \widetilde{G_{gw}} \cap \widetilde{M};R))^m\right)
      \\
      \to
      \bigoplus_{u \in V} K_0(\calh(\widetilde{G_{u}}\cap \widetilde{M};R))^m
    \end{multline*}
    is $K_0(\calh(\widetilde{G};R))$ because of
    Theorem~\ref{the:connective_case}~\ref{the:connective_case:q_is_0_Sub} and
    Remark~\ref{rem:computation_of_bredon_chain-complex} by the same argument as it
    appears in the proof of Lemma~\ref{lem:cokernel_of_beta_is_K_0(calh(G;R)} since
    $\bigl(\iota^* (|X| \times \IR)\bigr)^K$ is connected for every compact open subgroup
    $K$ of $\widetilde{G}$. It does not matter that $\iota^* (|X| \times \IR)$ is a $\widetilde{G}$-$CW$-complex
    but not a simplicial complex, since in the description of $\beta_{i,j}$ appearing in
    Remark~\ref{rem:computation_of_bredon_chain-complex} the case $j_i(+) = j_-(i)$ never occurs.
    
    We can identify
    $\bigoplus_{v \in V} K_0(\calh(\widetilde{G_{v}}\cap \widetilde{M};R))$ and the cokernel of 
    $\delta$,  since  we have the exact sequence
    $A^m \xrightarrow{\pi - \id} A^m \xrightarrow{\alpha} A \to 0$ for every abelian group $A$.
    The cokernel of $\delta \oplus \epsilon$ is isomorphic the cokernel of the composite
    of $\epsilon$ with the map
    \[\bigoplus_{v \in V}\operatorname{aug} \colon
      \bigoplus_{v \in V} K_0(\calh(\widetilde{G_{v}}\cap \widetilde{M};R))^m
      \to \bigoplus_{v \in V} K_0(\calh(\widetilde{G_{v}}\cap \widetilde{M};R)) = \cok(\delta).
     \]
    For every
    $(v,w) \in E$, $g \in F(v,w)$, and $u \in V$ the diagram
\[
  \xymatrix@!C=16em{K_0(\calh(\widetilde{G_{v}} \cap \widetilde{G_{gw}} \cap \widetilde{M};R))^m
    \ar[r]^-{\epsilon_{(v,w),u}}
    \ar[d]_{\operatorname{aug}}
    &
    K_0(\calh(\widetilde{G_{u}}\cap \widetilde{M};R))^m
    \ar[d]^{\operatorname{aug}}
    \\
    K_0(\calh(\widetilde{G_{v}} \cap \widetilde{G_{gw}} \cap \widetilde{M};R))
    \ar[r]_-{\gamma_{(v,w),u}}
    &
    K_0(\calh(\widetilde{G_{u}}\cap \widetilde{M};R))
  }
\]
commutes, since $\alpha \circ \pi = \alpha$ holds.  This
finishes the proof of Lemma~\ref{lem:cokernel_of_beta_is_K_0(calh(widetilde(G);R)}.
\end{proof}


\section{The projective class group  of the Hecke algebras of $\SL_n(F)$, $\PGL_n(F)$ and $\GL_n(F)$}%
\label{sec:K_0-SL-Gl-PGL}

\typeout{---- Section 6: The projective class group  of the Hecke algebras of $\SL_n(F)$, $\PGL_n(F)$ and $\GL_n(F)$ --}

Next we apply the recipes of
Sections~\ref{sec:The_main_recipe_for_the_computation_of_the_projective_class_group} to
some prominent reductive $p$-adic groups $G$ as an illustration.
For the remainder of this section $R$ is a uniformly regular  ring with $\IQ \subseteq R$.

Note that for a reductive $p$-adic groups $G$ the assembly map
$H_n^G(\EGF{G}{\COP};\bfK_R) \to H_n^G(G/G;\bfK_R) = K_n(\calh(G;R))$ is bijective for all
$n \in \IZ$ by Theorem~\ref{the:FJC_for_Hecke_algebras}. Moreover, the Bruhat-Tits
building $X$ of $G$ or of $G/\cent(G)$ can serve as the desired simplicial complex $X$
appearing in
Section~\ref{sec:The_main_recipe_for_the_computation_of_the_projective_class_group}.  The
original construction of the Bruhat-Tits building can be found
in~\cite{Bruhat-Tits(1972)}.  For more information about buildings we refer
to~\cite{Abramenko-Brown(2008), {Bridson-Haefliger(1999)}, Brown(1998), Ronan(1989)}.  The
space $X$ carries a CAT(0)-metric, which is invariant under the action of $G$ or
$G/\cent(G)$, see~\cite[Theorem~10A.4 on page~344]{Bridson-Haefliger(1999)}, Hence $|X|^H$
is contractible for any compact open subgroup $H$ of $G$ or $G/\cent(G)$, since $X^H$ is a
convex non-empty  subset of $X$ and hence contractible by~\cite[Corollary~II.2.8 on
page~179]{Bridson-Haefliger(1999)}. Therefore the geometric realization of the 
Bruhat-Tits building $X$ is (after possibly subdividing to achieve a cellular action) a
model for $\EGF{G}{\COP}$ or of $\EGF{G/\cent(G)}{\COP}$.


\subsection{$\SL_n(F)$}\label{sec:SL_n(F)}
   
We begin with computing $K_0(\calh(\SL_n(F);R))$, where $F$ is a non-Archimedean local
field with valuation $v \colon F \to \IZ \cup \{ \infty \}$.  The following claims about
the Bruhat-Tits building $X$ for $\SL_n(F)$ (and later about $X'$) can all be verified
from the description of $X$ in~\cite[Sec.~6.9]{Abramenko-Brown(2008)}. 
  
For $l=0,\ldots,n-1$ let $U^{\text{S}}_l$ be the compact open subgroup of $\SL_n(F)$
consisting of all matrices $(a_{ij})$ in $\SL_n(F)$ satisfying $v(a_{i,j}) \geq -1$ for
$1 \leq i \leq n - l < j \leq n$, $v(a_{i,j}) \geq 1$ for $1 \leq j \leq n-l  < i \leq n$ and
$v(a_{i,j}) \geq 0$ for all other $i,j$.   In particular $U^{\text{S}}_0 = \SL_n(\calo)$,
where $\calo = \{z \in F \mid v \geq 0 \}$.  The intersection of the $U^{\text{S}}_l$-s is
the Iwahori subgroup $I^{\text{S}}$ of $\SL_n(F)$. It is given by those matrices $A$ in $\SL_n(F)$
for which $v(a_{i,j}) \ge 1$ for $i > j$ and $v(a_{i,j}) \ge 0$ for $i \le  j$ hold.
  
The $(n-1)$-simplex $\Delta$ can be chosen with an ordering on its vertices such that the
isotropy group of its $l$-th vertex $v_l$ is $U^{\text{S}}_l$.  The isotropy group of a
face $\sigma$ of $\Delta$ is the intersection of the isotropy groups of the vertices of
$\sigma$.  In particular, the isotropy group of $\Delta$ is the Iwahori subgroup $I^S$ of
$\SL_n(F)$.  Consider the map
\begin{equation*}
d^{\SL_n(F)} \colon \bigoplus_{0 \leq i < j \leq n-1} K_0 ( \calh(U_i^\text{S} \cap U_j^\text{S};R))
  \to \bigoplus_{0 \leq l \leq n-1}  K_0 ( \calh(U_l^\text{S};R)),
\end{equation*}
for which the component
$d^{\SL_n(F)}_{i<j,l} \colon K_0 ( \calh(U_i^\text{S} \cap U_j^\text{S};R)) \to
K_0 (\calh(U_l^\text{S};R))$ is given by $-K_0(\calh(f_{i<j}^i;R))$, if $l = i$,
by $K_0(\calh(f_{i<j}^j;R))$, if $l = j$, and is zero, if $l \notin\{i,j\}$, where
$f_{i<j}^{k} \colon U^\text{S}_i \cap U^\text{S}_j \to U^\text{S}_k$ is the inclusion for
$k =i,j$.

Then the cokernel of $d^{\SL_n(F)}$ is  $K_0(\calh(\SL_n(F);R))$ by
Lemma~\ref{lem:cokernel_of_beta_is_K_0(calh(G;R)}
and Remark~\ref{rem:lem:cokernel_of_beta_is_K_0(calh(G;R)_strict_fundamental_domain}.


\subsection{$\PGL_n(F)$}\label{sec:PGL_n(F)}

Next we compute $K_0(\calh(\PGL_n(F);R))$.  The action of $\SL_n(F)$ on
$X$ extends to an action of $\GL_n(F)$.  This action factors through the canonical projection
$\pr \colon \GL_n(F) \to \PGL_n(F)$ to an action of $\PGL_n(F)$.
These actions are still simplicial, but no longer cellular.  Let
\begin{equation*}
  \widehat{h} := \left(
    \begin{array}{cccc}  & 1 & &  
      \\
     &  & \ddots  \\ & & & 1 \\ \zeta
     \end{array} \right) \in \GL_n(F)
\end{equation*}
where we chose a uniformizer $\zeta \in F$, i.e., an element in $F$ satisfying $v(\zeta) = 1$.
Obviously  $\widehat{h}^n$ is the diagonal matrix $\zeta \cdot I_n$,
all whose diagonal entries are $\zeta$, and hence is central
 in $\GL_n(F)$.  Define $h \in \PGL_n(F)$ by $h = \pr(\widehat{h})$.
 Then $h v_l = v_{l+1}$ for $l = 0, \ldots, n-2$ and  $hv_{n-1} =  v_0$
 and $h^n$ is the unit in $\PGL_n(F)$. In
particular, the action of $\PGL_n(F)$ is transitive on the vertices of $X$.  To obtain a
cellular action, $X$ can be subdivided to $X'$ as follows.  The $(n - 2)$-skeleton
of $X$ is unchanged, while the $(n-1)$-simplices of $X$ are in $X'$ replaced
with cones on their boundary.  More formally, the vertices of $X'$ are the vertices of $X$
and the barycenters $b_\sigma$ of $(n-1)$-simplices $\sigma$ of $X$.  A set $S$ of vertices of $X'$
is a simplex of $X'$, if and only if $S$ is a $k$-simplex of $X$ and $k<n-1$ or if $S$
contains exactly one barycenter $b_\sigma$ and for all $v \in S \setminus \{ b_\sigma \}$
are vertices of $\sigma$ (in the simplicial structure of $X$).  The action of $\PGL_n(F)$
on $X'$ is then cellular and is transitive on $(n-1)$-simplices of $X'$.  There are two
orbits of vertices, represented by $v_0$ and $b_\Delta$.  Let $k := \lfloor n/2 \rfloor$.
There are $k+1$ orbits of $1$-simplices, represented by $\{v_0,v_1\}$,\dots,$\{v_0,v_k\}$
and $\{v_0,b_\Delta\}$.  
Next we describe some isotropy groups.
      
For an open subgroup $W \subseteq \PGL_n(F)$ we denote by $\widetilde{W}$
its preimage under the projection $\pr \colon \GL_n(F) \to \PGL_n(F)$.
For $l=0,\dots,n-1$ let $U^{\text{G}}_l$ be the compact open subgroup of $\GL_n(F)$
given by $\widehat{h}^l\GL_n(\calo)\widehat{h}^{-l} = \widetilde{\PGL_n(F)_{v_l}} =\widetilde{\PGL_n(F)_{h_lv_0}}$
In particular $U^{\text{G}}_0 = \GL_n(\calo)$. Note that
\[
  U_l^{\text{G}} \cap \SL_n(F)
  = (\widehat{h}^{l}\GL_n(\calo)\widehat{h}^{-l}) \cap \SL_n(F)
  = \widehat{h}^{l}\SL_n(\calo)\widehat{h}^{-l} = U_l^S
\]
holds. The intersection of the $U^{\text{G}}_l$-s is the Iwahori subgroup $I^{\text{G}}$
of $\GL_n(F)$.  Let $U^{\text{P}}_l$ be the image of $U^{\text{G}}_l$ in $\PGL_n(F)$. This
is the isotropy groups of the vertex $v_l$ for the action of $\PGL_n(F)$.  The Iwahori
subgroup $I^{\text{P}}$ of $\PGL_n(F)$ is the image of $I^{\text{G}}$ under $\pr$. It is
the pointwise isotropy subgroup for $\Delta$.  Let $H$ be the subgroup generated by the
image of $h$ in $\PGL_n(F)$.  It is a cyclic subgroup of order $n$ that cyclically
permutes the vertices of $\Delta$.  This subgroup normalizes $I^{\text{P}}$ and the
isotropy group of $b_\Delta$ is the product $H I^{\text{P}}$.  Recall that $v_l = h^lv_0$ and hence
$U_l^P = h^lU_0Ph^{-l}$

Write
$i_H \colon I^{\text{P}} \to H I^{\text{P}}$, $i_0 \colon I^{\text{P}} \to U_0^\text{P}$,
$c_0 \colon U_0^\text{P} \cap U_i^\text{P} \to U_0^\text{P}$ for the inclusions and define
$c_l \colon U_0^\text{P} \cap U_l^\text{P} \to U_0^\text{P}$ by $ z \mapsto h^{-l}zh^l$.
Let
\begin{multline*}
  d^{\PGL_n(F)} \colon K_0( \calh(I^{\text{P}};R)) \oplus
  \bigoplus_{l =1}^k  K_0 ( \calh(U_0^\text{P} \cap U_l^\text{P};R)) \\
\to K_0( \calh(H I^{\text{P}};R)) \oplus K_0 ( \calh(U_0^\text{P};R))
\end{multline*}
be the map that is $K_0(i_H)  \times - K_0(i_0)$ on $K_0( \calh(I^{\text{P}};R))$ and
$0 \times (K_0(c_l) - K_0(c_0))$ on
$K_0( \calh(U_0^{\text{P}} \cap U_l^{\text{P}};R))$.
The cokernel of the homomorphism $d^{\PGL_n(F)}$ agrees with
$S\!H_0^{\PGL_n(F)}\bigl(X';K_0(\calh(?;R))\bigr)$ by Lemma~\ref{lem:cokernel_of_beta_is_K_0(calh(G;R)},
if , using the notation of Section~\ref{subsec:The_general_case},
we put  $E = \{v_0,b_{\Delta}\}$ with $ v_0 < b_{\Delta}$, $F(v_0.v_0) = \{h,h^2, \ldots, h^k\}$,
and $F(v_0,b_{\Delta}) = \{e\}$.


\subsection{$\GL_n(F)$}\label{sec:GL_n(F)}

Next we compute $K_0(\calh(\GL_n(F);R))$. Note that $\GL_n(F)$ has a non-compact center. Hence
Subsection~\ref{subsec:The_general_case} does not apply and we have to pass to the setting
of Subsection~\ref{subsec:A_variation} using the short exact sequence
$1 \to C = \cent(\GL_n(F)) \to \GL_n(F) \xrightarrow{\pr} \PGL_n(F) \to 1$, the
discussion in Subsection~\ref{sec:PGL_n(F)} and
Lemma~\ref{lem:cokernel_of_beta_is_K_0(calh(widetilde(G);R)}.

Let $\widetilde{M}$ be the kernel of the composite
$\mu \colon \GL_n(F) \xrightarrow{\det} F^{\times}\xrightarrow{\nu}  \IZ$.
Let $\widehat{H} \subseteq \GL_n(F)$ be the infinite cyclic subgroup
generated by the element $\widehat{h}$. Note  that $\widetilde{M} \cap C$ consists
of those diagonal matrices whose entries on the diagonal are all the same and are sent to $0$ under $\nu$.
We conclude  $(\GL_n(\calo) \cdot C) \cap\widetilde{M} = \GL_n(\calo)$ from
$C \cap \widetilde{M} \subseteq \GL_n(\calo) \subseteq \widetilde{M}$.
Recall that for $W \subseteq \PGL_n(F)$ we denote by $\widetilde{W}$
its preimage under $\pr \colon \GL_n(F) \to \PGL_n(F)$.
Since $\pr(U_l^{\text{G}}) = U_l^{\text{P}}$, we get for $l = 0, \ldots, n-1$
\begin{multline*}
 \widetilde{U_l^{\text{P}}} \cap \widetilde{M}
  = (U_l^{\text{G}} \cdot C) \cap \widetilde{M}
  =  (\widehat{h}^l\GL_n(\calo)\widehat{h}^{-l} \cdot C) \cap\widetilde{M}
  \\
  = \widehat{h}^l\bigl((\GL_n(\calo) \cdot C) \cap\widetilde{M}\bigr)\widehat{h}^{-l}
  = \widehat{h}^l\GL_n(\calo)\widehat{h}^{-l}
  = U_l^{\text{G}}.
\end{multline*}
Now one easily checks $\widetilde{I^{\text{P}}} \cap \widetilde{M} = I^{\text{G}}$. Finally we show
$\widetilde{HI^{\text{P}}} \cap \widetilde{M} = I^{\text{G}}$. We get
$I^{\text{G}} \subseteq \widetilde{HI^{\text{P}}} \cap \widetilde{M}$ from
$\widetilde{I^{\text{P}}} \cap \widetilde{M} = I^{\text{G}}$. Consider an element
$A \in \widetilde{HI^{\text{P}}} \cap \widetilde{M}$.  We can find an integer $b$, an element
$B \in I^{\text{G}}$, and an element $D\in C$ such that $A = \widehat{h}^bBD$ and $\nu(A) = 0$
holds. From $I^{\text{G}} \subseteq \widetilde{M}$  we conclude
$\widehat{h}^bD \in \widetilde{M}$. Since $\mu(D)$ is divisible by $n$ and
$\mu(\widehat{h}) =1$ holds, $b$ is divisible by $n$. This implies
$\widehat{h}^b \in C$ and hence $\widehat{h}^bD\in C \cap \widetilde{M}$. 
As $(C \cap \widetilde{M})I^{\text{G}} = I^{\text{G}}$ holds, we conclude $A \in I^{\text{G}}$. Hence
$\widetilde{HI^{\text{P}}} \cap \widetilde{M} = I^{\text{G}}$ holds.

Let
$\widetilde{i}_0 \colon I^{\text{G}}  \to U_0^{\text{G}}$ and 
$\widetilde{c}_0 \colon U_0^{\text{G}}\cap U_i^{\text{G}} \to U_0^{\text{G}}$
be the inclusions and let $\widetilde{c}_l \colon U_0^{\text{G}} \cap U_l^{\text{G}}  \to U_0^{\text{G}}$
be the map  sending $\widetilde{z}$ to $\widehat{h}^{-l}\widetilde{z}\widehat{h}^l$. Let 
\begin{multline*}
  \overline{d}^{\GL_n(F)} \colon K_0( \calh(I^{\text{G}};R)) \oplus
  \bigoplus_{l = 1}^{k} K_0 ( \calh(U_0^{\text{G}} \cap  U_l^{\text{G}};R))
  \\
  \to  
   K_0( \calh (I^{\text{G}} ;R)) \oplus K_0 ( \calh(U_0^{\text{G}};R))
 \end{multline*}
 be the map that is $\id_{K_0(I^{\text{G}})} \times -K_0(\widetilde{i}_0)$ on
 $K_0( \calh(I^{\text{G}};R))$ and
 $0 \times (K_0(\widetilde{c}_l)  - K_0(\widetilde{c}_0))$ on
 $K_0( \calh(U_0^\text{G}\cap U_i^\text{G};R))$.
 The cokernel of the map
 $\overline{d}^{\GL_n(F)}$  is $K_0(\calh(\GL_n(F);R))$
 by Lemma~\ref{lem:cokernel_of_beta_is_K_0(calh(widetilde(G);R)} 
 Let 
 \[
  \widetilde{d}^{\GL_n(F)} \colon 
  \bigoplus_{l = 1}^k  K_0 ( \calh(U_0^{\text{G}} \cap  U_l^{\text{G}};R))
  \to K_0 (\calh(U_0^{\text{G}};R))
\]
be the map which is given by $K_0(\widetilde{c}_l)  - K_0(\widetilde{c}_0)$ on
$K_0( \calh(U_0^\text{G}\cap U_l^{\text{G}};R))$.
Since  $\widetilde{d}^{\GL_n(F)} $ has the same cokernel as   $\overline{d}^{\GL_n(F)}$,
the cokernel of $\widetilde{d}^{\GL_n(F)}$ is
$K_0(\calh(\GL_n(F);R))$.


 \section{Homotopy colimits}\label{sec:homotopy-colimits}

 \typeout{-------------------- Section 7: Homotopy colimits --------------------}


\subsection{The Farrell-Jones assembly map as a map of homotopy colimits}%
\label{subsec:The_Farrell-Jones_assembly_map_asa_map_of_homotopy_colimits}

Next we want to extend the considerations of Section~\ref{sec:K_0-SL-Gl-PGL} to the higher
$K$-groups. For this purpose and the proofs appearing in~\cite
{Bartels-Lueck(2023K-theory_red_p-adic_groups)} it is worthwhile to write down the
assembly map in terms of homotopy colimits.  The projections $G/U \to G/G$ for $U$ compact
open in $G$ induce a map
\begin{equation}
  \hocolimunder_{G/U \in     \Or_\COP(G)} \bfK_R(G/U) \to \bfK_R(G/G) \simeq\bfK(\calh(G;R)).
  \label{assembly_as_homotopy_colimit}
\end{equation}
This map can be identified after applying $\pi_n$ with the assembly map appearing in
Theorem~\ref{the:Main_Theorem}~\ref{the:Main_Theorem:assembly} and
Theorem~\ref{the:FJC_for_Hecke_algebras}.  This follows
from~\cite[Section~5]{Davis-Lueck(1998)}.


\subsection{Simplifying the source of the  Farrell Jones assembly map}%
\label{subsec:Simplifying_the_source_of_the__Farrell-Jones_assembly_map}

Let $X$ be an abstract simplicial complex with  simplicial $G$-action such 
that the isotropy group of each vertex is
compact open and  the $G$-action is cellular. Furthermore we  assume that $|X|^K$
is weakly contractible for any compact open subgroup of $G$. Then $|X|$ is a model for
$\EGF{G}{\COP}$.

Let $C$ be a collection of simplices of $X$ that contains at least one simplex from each
orbit of the action of $G$ on the set of simplices of $X$.  Define a category $\calc(C)$
as follows.  Its objects are the simplices from $C$.   A morphism
$gG_{\sigma} \colon \sigma \to \tau$ is an element $gG_{\sigma} \in G/G_{\sigma}$
satisfying $g\sigma \subseteq \tau$. The composite of $gG_{\sigma} \colon \sigma \to \tau$ with
$hG_{\tau} \colon \tau \to \rho$ is $hgG_{\sigma} \colon \sigma \to \rho$. Define a
functor

\begin{equation}
  \iota_C \colon \calc(C)^{\op} \to \OrGF{G}{\COP}
  \label{iota_C}
\end{equation}
by sending an object $\sigma$ to $G/G_\sigma$ and a morphism
$gG_{\sigma} \colon \sigma \to \tau$ to
$R_{g} \colon G/G_{\tau} \to G/G_{\sigma}, \; g'G_{\tau}\mapsto g'gG_{\sigma}$.

\begin{lemma}\label{lem:homotopy_equivalence_induced_by_iota}
  Under the assumptions above the map induced by the functor $\iota_C$
  \[
    \hocolimunder_{\sigma \in \calc(C)^{\op}} \bfK_R(G/G_{\sigma}) \xrightarrow{\sim}
    \hocolimunder_{G/U \in \OrGF{G}{\COP}} \bfK_R(G/U)
  \]
  is a weak homotopy equivalence.
\end{lemma}
\begin{proof} We want to apply the criterion~\cite[9.4]{Dwyer-Kan(1984a)}.  So we have to
  show that the geometric realization of the nerve of the category $G/K\downarrow \iota_C$ is
  a contractible space for every object $G/K$ in $\OrGF{G}{\COP}$.  An object in
  $G/K \downarrow \iota_C$ is a pair $(\sigma, u)$ consisting of an element $\sigma \in C$
  and a $G$-map $u \colon G/K \to G/G_{\sigma}$. A  morphism $(\sigma, u) \to (\tau,v)$ in
  $G/K\downarrow \iota_C$ is given by a morphism $gG_{\tau} \colon \tau \to \sigma$
  in $\calc(C)$ such that the $G$-map $R_g \colon G/G_{\sigma} \to G/G_{\tau}$
  sending $zG_{\sigma}$ to $zgG_{\tau}$ satisfies  $v \circ R_g = u$.

  Let $\calp(X^K)$ be the poset given by the simplices of $X^K$ ordered by inclusion.
  Then we get an equivalence of categories
  \[F \colon \calp(X^K)^{\op} \xrightarrow{\simeq} G/K \downarrow \iota_C
  \]
  as follows. It sends a simplex $\sigma$ to the object
  $(\sigma, \pr_{\sigma} \colon G/K \to G/G_{\sigma})$ for the canonical projection
  $\pr_{\sigma}$. A morphism $\sigma \to  \tau$ in  $\calp(X^K)^{\op}$
  is sent to the morphism
  $(\sigma, \pr_{\sigma}) \to (\tau,\pr_{\tau})$ in $G/K \downarrow \iota_C$
  which is given by the morphism 
  $eG_{\tau} \colon \tau \to \sigma$ in $\calc(C)$.

  Consider an object $(\sigma, u)$ in $G/K \downarrow \iota_C$.  We
  want to show that it is isomorphic to an object in the image of $F$.
  Choose $g \in G$ such that $g^{-1}Kg \subseteq G_{\sigma}$ holds and
  $u$ is the $G$-map $R_g \colon G/K \to G/G_{\sigma}$ sending $zK$ to
  $zgG_{\sigma}$. Then $K \subseteq G_{g\sigma}$ and we can consider
  the object $F(g\sigma) = (g\sigma, \pr_{g\sigma})$ for the
  projection $\pr_{g\sigma} \colon G/K \to G_{g\sigma}$. Now the
  isomorphism $gG_{\sigma} \colon \sigma \to g\sigma$ in $\calc(C)$
  induces an isomorphism $F(g\sigma) \xrightarrow{\cong} (\sigma, u)$
  in $G/K \downarrow \iota_C$.

  Obviously $F$ is faithful. It remains to show that $F$ is full.  Fix
  two objects $\sigma$ and $\tau$ in $\calp(X^K)$.  Consider a
  morphism
  $f \colon F(\sigma) = (\sigma, \pr_{\sigma}) \to F(\tau) = (\tau,
  \pr_{\tau})$ in $G/K \downarrow \iota_C$.  It is given by a morphism
  $gG_{\tau} \colon \tau \to \sigma$ in $\calc(C)$ such that the
  composite of $R_g \colon G/G_{\sigma} \to G/G_{\tau}$ with
  $\pr_{\sigma}$ is $\pr_{\tau}$. This implies $gG_{\tau} = G_{\tau}$
  and hence $g \in G_{\tau}$.  Since $g\tau \subseteq \sigma$ holds by
  the definition of a morphism in $\calc(C)$, we get
  $\tau \subseteq \sigma$.  Hence $f$ is the image of the morphism
  $\sigma \to \tau$ under $F$. This shows that $F$ is full.

  Hence it remains to show that geometric realization of the nerve of $\calp(X^K)^{\op}$ is
  contractible.  Since this is the barycentric subdivision of $|X|^K$, this follows from the assumptions.
\end{proof}

Suppose additionally that $X$ admits a strict fundamental domain $\Delta$, i.e., a
simplicial subcomplex $\Delta$ that contains exactly one simplex from each orbit for the
$G$-action on the set of simplices of $X$.  Then we can take for $C$ the simplices from
$\Delta$.  In this case $\calc(C)$ can be identified with the poset $\calp(\Delta)$ of
simplices of $\Delta$.  Recall that for any open subgroup $U$ of $G$,
there is an explicit weak homotopy equivalence
$\bfK(\calh(U;R)) \xrightarrow{\simeq} \bfK_R(G/U)$, where the source is the 
$K$-theory spectrum $\bfK(\calh(U;R))$ of the Hecke algebra $\calh(U;R)$,
see~\cite[5.6 and Remark~6.7]{Bartels-Lueck(2023foundations)}.
 Lemma~\ref{lem:homotopy_equivalence_induced_by_iota} implies

 \begin{theorem}\label{the:homotopy_equivalence_induced_by_iota_strict_fundamental_domain}
   Let $X$ be an abstract simplicial complex with a simplicial $G$-action such that the isotropy group
   of each vertex  is compact open, the $G$-action is cellular, and $|X|^K$ is weakly
  contractible for every compact open subgroup $K$ of $G$. Let $\Delta$ be a strict
  fundamental domain. 

  Then the assembly map
  \begin{equation}\label{eq:COP-assembly-map-hocolim-cala-Delta}
    \hocolimunder_{\sigma \in \calp(\Delta)^{\op}} \bfK(\calh(G_{\sigma};R))
    \to \hocolimunder_{G/U \in \OrGF{G}{\COP}} \bfK_R(G/U)
  \end{equation}
  that is induced by the functor $\calp(\Delta)^{\op} \to \OrGF{G}{\COP}$ sending
  a simplex $\sigma$ to $G_{\sigma}$, is a weak homotopy equivalence,
\end{theorem}

\begin{example}[$\SL_n(F)$]\label{ex:SL} Let $X$ be the Bruhat-Tits building for
  $\SL_n(F)$.  Then the canonical $\SL_n(F)$ action on $X$ is cellular. We will use again the notation introduced in
  Section~\ref{sec:K_0-SL-Gl-PGL}.  The $(n-1)$-simplex $\Delta$, viewed as a
  subcomplex of $X$, is a strict fundamental domain.
  Applying this in the case $n=2$ yields the
  homotopy pushout diagram
  \begin{equation*}
    \xymatrix{\bfK(\calh(I^\text{S};R)) \ar[r] \ar[d] & \bfK(\calh(U_{1}^\text{S};R)) \ar[d]
      \\ 
      \bfK(\calh(U_{0}^\text{S};R)) \ar[r] & \bfK(\calh(\SL_2(F);R)).}
  \end{equation*}  
  For the K-groups this yields a Mayer-Vietoris sequence, infinite to the left,
  \begin{multline}
    \cdots \to K_n(\calh(I^\text{S};R))  \to K_n(\calh(U_{1}^\text{S};R)) \oplus  K_n(\calh(U_{0}^\text{S};R))  \to
    K_n(\calh(\SL_2(F);R))
    \\
    \to  K_{n-1}(\calh(I^\text{S};R))  \to K_{n-1}(\calh(U_{1}^\text{S};R)) \oplus  K_{n-1}(\calh(U_{0}^\text{S};R))  \to \cdots
    \\
    \cdots \to K_0(\calh(I^\text{S};R))  \to K_0(\calh(U_{1}^\text{S};R)) \oplus  K_0(\calh(U_{0}^\text{S};R))  \to
    K_0(\calh(\SL_2(F);R)) \to 0
    \label{MV_sequence_SL_n}
  \end{multline}
  and  $K_n(\calh(\SL_2(F);R)) = 0$ for $n \le -1$.
  
  For $n=3$ we obtain the homotopy push-out diagram

  \begin{equation*}
    \xymatrix@!C=5em{&  \bfK(\calh(U_{12}^\text{S};R)) \ar[rr] \ar[dd]|\hole
      & & \bfK(\calh(U_{2}^\text{S};R)) \ar[dd]
      \\
      \bfK(\calh(I^\text{S};R))  \ar[ur]\ar[rr]\ar[dd]
      & & \bfK(\calh(U_{02}^\text{S};R)) \ar[ur]\ar[dd]
      \\
      & \bfK(\calh(U_{1}^\text{S};R)) \ar[rr] |!{[r]}\hole
      & & \bfK(\calh(\SL_3(F);R)) \\
      \bfK(\calh(U_{01}^\text{S};R)) \ar[rr]\ar[ur]
      & & \bfK(\calh(U_{0}^\text{S};R)) \ar[ur]
    }
  \end{equation*} 
  where we abbreviated $U^\text{S}_{ij} := U^\text{S}_i \cap U^\text{S}_j$.  In general,
  for $\SL_n(F)$ we obtain a homotopy push-out diagram whose shape is an n-cube.

  To such an $n$-cube there is assigned a spectral sequence concentrated in the region for
  $p \ge 0$ and $0 \le q \le n-1$, which corresponds to the spectral sequence appearing in
  Theorem~\ref{the:Main_Theorem}~\ref{the:Main_Theorem:spectral_sequence}
\end{example}


\section{Allowing central characters  and actions on the coefficients}%
\label{sec:Allowing_central_characters_and_G-actions_on_the_coefficients}

\typeout{--- Section 9: Allowing central characters  and actions on the coefficients -------}

So far we have only considered the standard Hecke algebra $\calh(G;R)$.
There are more general Hecke
algebras $\calh(G;R,\rho,\omega)$, see~\cite{Bartels-Lueck(2022KtheoryHecke)},
and all the discussions of this paper carry over to them in the obvious way.


\def\cprime{$'$} \def\polhk#1{\setbox0=\hbox{#1}{\ooalign{\hidewidth
  \lower1.5ex\hbox{`}\hidewidth\crcr\unhbox0}}}


\begin{thebibliography}{10}

\bibitem{Abramenko-Brown(2008)}
P.~Abramenko and K.~S. Brown.
\newblock {\em Buildings}, volume 248 of {\em Graduate Texts in Mathematics}.
\newblock Springer, New York, 2008.
\newblock Theory and applications.

\bibitem{Bartels-Lueck(2022KtheoryHecke)}
A.~Bartels and W.~L{\"u}ck.
\newblock On the algebraic {$K$}-theory of {H}ecke algebras.
\newblock preprint, arXiv:2204.07982 [math.KT], to appear in the Festschrift
  \emph{Mathematics Going Forward}, Lecture Notes in Mathematics~2313,
  Springer, 2022.

\bibitem{Bartels-Lueck(2023K-theory_red_p-adic_groups)}
A.~Bartels and W.~L{\"u}ck.
\newblock Algebraic {$K$}-theory of reductive $p$-adic groups.
\newblock Preprint, arXiv:2306.03452 [math.KT], 2023.

\bibitem{Bartels-Lueck(2023foundations)}
A.~Bartels and W.~L{\"u}ck.
\newblock Inheritance properties of the {$K$}-theoretic {F}arrell-{J}ones
  {C}onjecture for totally disconnected groups.
\newblock Preprint arXiv:2306.01518 [math.KT], 2023.

\bibitem{Bernstein(1992)}
J.~Bernstein.
\newblock Draft of: Representations of $p$-adic groups.
\newblock
  http//www.math.tau.ac.il/~bernstei/Unpublished\_texts/Unpublished\_list.html,
  1992.

\bibitem{Bridson-Haefliger(1999)}
M.~R. Bridson and A.~Haefliger.
\newblock {\em Metric spaces of non-positive curvature}.
\newblock Springer-Verlag, Berlin, 1999.
\newblock Die Grundlehren der mathematischen Wissenschaften, Band 319.

\bibitem{Brown(1998)}
K.~S. Brown.
\newblock {\em Buildings}.
\newblock Springer-Verlag, New York, 1998.
\newblock Reprint of the 1989 original.

\bibitem{Bruhat-Tits(1972)}
F.~Bruhat and J.~Tits.
\newblock Groupes r\'eductifs sur un corps local.
\newblock {\em Inst. Hautes \'Etudes Sci. Publ. Math.}, 41:5--251, 1972.

\bibitem{Dat(2003)}
J.-F. Dat.
\newblock Quelques propri\'et\'es des idempotents centraux des groupes
  {$p$}-adiques.
\newblock {\em J. Reine Angew. Math.}, 554:69--103, 2003.

\bibitem{Dat(2007)}
J.-F. Dat.
\newblock Th\'eorie de {L}ubin-{T}ate non-ab\'elienne et repr\'esentations
  elliptiques.
\newblock {\em Invent. Math.}, 169(1):75--152, 2007.

\bibitem{Davis-Lueck(1998)}
J.~F. Davis and W.~L{\"u}ck.
\newblock Spaces over a category and assembly maps in isomorphism conjectures
  in ${K}$- and ${L}$-theory.
\newblock {\em $K$-Theory}, 15(3):201--252, 1998.

\bibitem{Dwyer-Kan(1984a)}
W.~G. Dwyer and D.~M. Kan.
\newblock A classification theorem for diagrams of simplicial sets.
\newblock {\em Topology}, 23(2):139--155, 1984.

\bibitem{Garrett(2012)}
P.~Garrett.
\newblock Smooth representations of totally disconnected groups.
\newblock unpublished notes, https://www-users.cse.umn.edu/\textasciitilde
  garrett/m/v/smooth\_of\_td.pdf, 2012.

\bibitem{Lueck(1989)}
W.~L{\"u}ck.
\newblock {\em Transformation groups and algebraic ${K}$-theory}, volume 1408
  of {\em Lecture Notes in Mathematics}.
\newblock Springer-Verlag, Berlin, 1989.

\bibitem{Lueck(2002b)}
W.~L{\"u}ck.
\newblock Chern characters for proper equivariant homology theories and
  applications to ${K}$- and ${L}$-theory.
\newblock {\em J. Reine Angew. Math.}, 543:193--234, 2002.

\bibitem{Lueck(2002d)}
W.~L{\"u}ck.
\newblock The relation between the {B}aum-{C}onnes conjecture and the trace
  conjecture.
\newblock {\em Invent. Math.}, 149(1):123--152, 2002.

\bibitem{Ronan(1989)}
M.~Ronan.
\newblock {\em Lectures on buildings}.
\newblock Academic Press Inc., Boston, MA, 1989.

\bibitem{Switzer(1975)}
R.~M. Switzer.
\newblock {\em Algebraic topology---homotopy and homology}.
\newblock Springer-Verlag, New York, 1975.
\newblock Die Grundlehren der mathematischen Wissenschaften, Band 212.

\end{thebibliography}


\end{document}